\documentclass[11pt,twoside]{amsart}

\usepackage[margin=1.4in]{geometry}
\usepackage{fancyhdr}

\usepackage{amsmath, amssymb, amsthm, amscd}
\usepackage{mathtools}
\usepackage{mathrsfs} 
\usepackage{enumitem}
\usepackage{comment}
\usepackage{tikz}
\usepackage{tikz-cd}

\usepackage[backref=page]{hyperref}
\usepackage[dvipsnames]{xcolor}
\hypersetup{colorlinks=true,citecolor=NavyBlue,linkcolor=BrickRed,urlcolor=Orange}
\usepackage[nameinlink,capitalize]{cleveref}

\theoremstyle{plain}
\newtheorem{theorem}{Theorem}[section]
\newtheorem{claim}[theorem]{Claim}
\newtheorem{proposition}[theorem]{Proposition}
\newtheorem{lemma}[theorem]{Lemma}
\newtheorem{corollary}[theorem]{Corollary}

\theoremstyle{definition}
\newtheorem{definition}[theorem]{Definition}
\newtheorem{example}[theorem]{Example}

\theoremstyle{remark}

\newtheorem*{remark}{Remark}

\newcommand{\theoremref}[1]{\hyperref[#1]{Theorem~\ref*{#1}}}
\newcommand{\lemmaref}[1]{\hyperref[#1]{Lemma~\ref*{#1}}}
\newcommand{\definitionref}[1]{\hyperref[#1]{Definition~\ref*{#1}}}
\newcommand{\propositionref}[1]{\hyperref[#1]{Proposition~\ref*{#1}}}
\newcommand{\conjectureref}[1]{\hyperref[#1]{Conjecture~\ref*{#1}}}
\newcommand{\corollaryref}[1]{\hyperref[#1]{Corollary~\ref*{#1}}}
\newcommand{\exampleref}[1]{\hyperref[#1]{Example~\ref*{#1}}}

\newcommand{\C}{\mathbb{C}}
\newcommand{\Z}{\mathbb{Z}}
\newcommand{\Q}{\mathbb{Q}}

\newcommand{\D}{\mathscr{D}}
\newcommand{\A}{\mathcal{A}}

\newcommand{\M}{\mathcal{M}}

\newcommand{\strucsheaf}{\mathcal{O}}
\newcommand{\proj}{\mathbb{P}}

\newcommand{\alb}{\text{Alb}}
\newcommand{\fmd}{\operatorname{FM}^{\operatorname{D}}}
\newcommand{\fm}{\operatorname{FM}}
\newcommand{\tfm}{\widetilde{\operatorname{FM}}}
\newcommand{\MHM}{\operatorname{MHM}}

\DeclareMathOperator{\sheafhom}{\mathcal{H\kern -3pt o\kern -2pt m\kern -1pt}}

\newcommand{\IM}{\operatorname{Im}}

\newcommand{\supp}{\operatorname{Supp}}
\newcommand{\codim}{\operatorname{codim}}
\newcommand{\defeq}{\mathrel{\mathop:}=}

\newcommand{\id}{\operatorname{id}}

\newcommand{\gr}{\operatorname{gr}}

\title[Generic vanishing of Dolbeault cohomology]{Generic vanishing of Dolbeault cohomology of mixed Hodge modules on abelian varieties}

\author[Ze Yun]{Ze Yun}
\address{Department of Mathematics, Stony Brook University}
\email{ze.yun@stonybrook.edu}

\date{\today}

\begin{document}

\begin{abstract}
    We study generic vanishing for the Dolbeault cohomology of mixed Hodge modules, where the differentials in the Dolbeault complex are modified by a one-form. We prove that, on the Dolbeault moduli space of Higgs line bundles, the cohomology support loci of mixed Hodge modules agree with the images of the cohomology support loci of their underlying $\mathscr{D}$-modules on the de Rham moduli space under the map induced by the nonabelian Hodge correspondence. We further apply this result to study Fourier-Mukai transforms of mixed Hodge modules on abelian varieties. We show that the extended Fourier-Mukai transform for bounded complexes of mixed Hodge modules has cohomology sheaves that are flat over the twistor line of $\lambda$-connections. This also gives an extension of Fourier transform of mixed Hodge modules to a compactification of the de Rham moduli space.
\end{abstract}

\maketitle

\section{Introduction}

\subsection{Agenda}
Generic vanishing for cohomology of topologically trivial line bundles on compact Kähler manifolds was first studied by Green-Lazarsfeld \cite{GreenLazarsfeldI} \cite{GreenLazarsfeldII}. For a compact Kähler manifold $X$, they showed the following generic vanishing theorem:
\begin{theorem}[Generic vanishing theorem \cite{GreenLazarsfeldI}]
    Let $X$ be a compact Kähler manifold and $a:X\rightarrow A$ be the Albanese map of $X$. the cohomology of a generic line bundle $L\in \text{Pic}^0 (X)$ vanishes in degrees less than $\dim a(X)$.
\end{theorem}
Even though the structure sheaf of $X$ in general doesn't have vanishing cohomology, a generic Chern class zero line bundle does satisfy vanishing in certain degrees. Moreover, they studied the geometry of the following subsets of $\operatorname{Pic}^0(X):$
$$S^i(X)\defeq \{L\in\operatorname{Pic}^0(X): H^i(X,L)\neq 0 \}.$$

\begin{theorem}[\cite{GreenLazarsfeldI} \cite{GreenLazarsfeldII}]
    $S^i(X)$ are finite unions of translates of subtori of $\operatorname{Pic}^0(X)$.
\end{theorem}

Beauville and Catanese \cite{Beauville1992} conjectured that the translates are always by torsion points, and this was proved by Simpson \cite{SimpsonmoduliofRankOne}. There is also a positive characteristic proof \cite{PinkRoessler}. A more general statement about torsion translates for mixed Hodge modules with integral structures on abelian varieties was proved by Schnell \cite{torsionpoints}.

Generic vanishing theory has various applications on geometry and topology of irregular varieties. For example, it follows directly from the generic vanishing theorem that for a variety with maximal dimension, $\chi(X,\omega_X)\geq 0$ \cite[p.390]{GreenLazarsfeldI}. It also leads to a numerical characterization for varieties birational to abelian varieties \cite{chenhaconcharacterization}. For other applications, see \cite{SimpsonmoduliofRankOne}, \cite{ArapuraGreenLazarsfeldSets}, \cite{EinLazarsfeldsingularitiesofTheta}, \cite{chenHaconAlgebraicFiberSpace}, \cite{chenhaconcharacterization}, \cite{Pareschi}, \cite{popaschnell}, \cite{Zerosofoneforms}.

Later Hacon \cite{Hacongenericvanishing} found another beautiful proof of the generic vanishing theorem, via Fourier-Mukai transforms and properties of higher direct images of dualizing sheaves $R^ia_* \omega_X$. It was further explored via the Albanese map and decomposition theorem of Hodge modules in \cite{popaschnell} (there's also a recent generalization to singular varieties \cite{duc2024genericvanishingsingularvarieties}). These studies and techniques also leads to the celebrated work on zeros of one-forms on varieties of general type \cite{Zerosofoneforms}. After the work of Popa and Schnell \cite{popaschnell}, it was clear that generic vanishing theory is deeply connected with properties of mixed Hodge modules on abelian varieties, and the study of them can be done via Fourier-Mukai transform of $\D$-modules developed by Laumon and Rothstein \cite{LaumonFouriertransform} \cite{Rothstein}. In \cite{holonomicDmodules}, Schnell systematically studied holonomic $\D$-modules on abelian varieties, their cohomology support loci and Fourier-Mukai transforms. There were also studies about generic vanishing phenomenon of cohomology of perverse sheaves on abelian varieties \cite{kramergenericvanishing}, and perverse sheaves on semi-abelian varieties \cite{mellinpropagation} \cite{perversesemiabelian}. In this paper, we continue the study of mixed Hodge modules on abelian varieties, along the lines of \cite{holonomicDmodules}. We study generic vanishing phenomenon for some coherent objects associated to $\D$-modules (the graded de Rham complexes), and apply it to study mixed Hodge modules on abelian varieties.

In \cite{popaschnell}, the generic vanishing for Dolbeault cohomologies (defined later) of pure Hodge modules of geometric origin was studied using the results of \cite{GreenLazarsfeldI}, \cite{GreenLazarsfeldII}, \cite{ArapuraGreenLazarsfeldSets}, and \cite{Simpson1992} on cohomology of harmonic bundles of rank one. It is natural to wonder whether part of their results is a general phenomenon for mixed Hodge modules on abelian varieties, and we show this is indeed the case, by using the technique of Fourier-Mukai transform, the properties of mixed Hodge modules on abelian varieties, and results in \cite{holonomicDmodules} about holonomic $\D$-modules on abelian varieties. As an application, our result also gives an algebraic proof of equality of dimensions of Dolbeault and de Rham cohomologies, see \corollaryref{cor:equality of dimension}, similar to \cite[Lemma 2.2]{Simpson1992} in rank one, without using harmonic metrics.

The motivation for this work is to study Fourier-Mukai transforms of mixed Hodge modules on abelian varieties. It was proved in \cite{LaumonFouriertransform} and \cite{Rothstein} that the derived category of coherent $\D$-modules on an abelian variety $A$ is equivalent to the derived category of coherent sheaves on $A^\natural$ (the moduli space of flat line bundles on $A$). Then one could study $\D$-modules on abelian varieties by studying coherent sheaves on $A^\natural$, and ask what kind of properties of sheaves on $A^\natural$ reflect properties of $\D$-modules on $A$. Holonomicity was considered in \cite{holonomicDmodules}. \cite{holonomicDmodules} extended the construction of Laumon and Rothstein, and the extended Fourier-Mukai transform for filtered $\D$-modules was constructed, by taking into account the filtration and $\lambda$-connections, and it lives over the twistor line $\mathbb{A}^1_\lambda$ of $\lambda$-connections. The transform of Laumon and Rothstein are sheaves on the fiber $\lambda=1$. It's expected that for mixed Hodge modules, equipped with the Hodge filtration, the Fourier-Mukai transform behaves nicely along the twistor line. And one hopes that adding the information of a filtration that behaves well will give further refinements for transforms of mixed Hodge modules. Our result give some applications along this line, see section 7 and section 8.


\textbf{Notation and convention of mixed Hodge modules.} In this paper, $A$ will denote an abelian variety, and $\hat{A}$ is the dual abelian variety. We will use $V=H^0(A,\Omega_A^1)$ to denote the space of $1$-forms on $A$. We consider $\D$-modules as quasi-coherent sheaves admitting a flat connection. Projective bundle associated to a vector bundle is the projective bundle of lines in that vector bundle. Given a filtered $\D$-module $(\M,F)$ on $X$, $\gr^F_\bullet\M$ means the graded $\operatorname{Sym}T_X$-module on $X$, while $\gr^F\M$ denotes the associated quasi-coherent sheaf on the cotangent bundle $T^* X$. Note that the theory of complex mixed Hodge modules and their derived category are still being developed \cite{MHMproject}, so we restrict ourselves to use the derived category of $\Q$-mixed Hodge modules constructed by \cite{SaitoMHM}. However, it is expected all the results in this paper to be true for complex mixed Hodge modules and their derived category, once the theory is fully developed, since our proofs only use some formal properties of their derived category and tensor product. Nevertheless, the same proof for our main theorem about cohomology support loci works for pure complex Hodge modules (see \theoremref{thm:pure case}), since the main ingredient we use in the proof is the decomposition theorem, and the decomposition theorem for pure complex Hodge modules is fully established in \cite{MHMproject}.

Starting from section 5, whenever we need to take tensor product of a mixed Hodge module with a unitary rank connection $(L_{\chi},\nabla)$, by \theoremref{thm:rational translates}, we can reduce to the case that there's a unitary $\Q$-VHS containing $(L_{\chi},\nabla)$ as a direct summand. This extra technical operation is due to the fact that we are using derived category of mixed Hodge modules defined over $\Q$, and we are only allowed to take tensor product with $\Q$-VHS in this category. Since all of our statements hold true under taking direct summands, adding extra summands in the proof doesn't affect our conclusion. We will explain this technical point again in section 5, see the remark before \theoremref{thm:pure case}.

\subsection{The main theorem}

We define the main object that we will study, and state our main theorem in this section.

Recall that a mixed Hodge module on a smooth variety $X$ of dimension $g$ is a filtered $\D$-module, meaning there is an increasing filtration of $\M$ by $\strucsheaf_X$-coherent subsheaves $F_k \M$, such that $F_l \D_X\cdot F_k \M\subset F_{k+l} \M$ ($F_l\D_X$ is the order filtration on differential operators). The filtration on $\M$ induces a filtration on the de Rham complex of $\mathcal{M}$ by the formula
$$
F_k \mathrm{DR}_X(\mathcal{M})=\left[F_k \mathcal{M} \rightarrow \Omega_X^1 \otimes F_{k+1} \mathcal{M} \rightarrow \cdots \rightarrow \Omega_X^g \otimes F_{k+g} \mathcal{M}\right].
$$
The associated graded complex for this filtration is
$$
\operatorname{gr}_k^F \mathrm{DR}_X(\mathcal{M})=\left[\operatorname{gr}_k^F \mathcal{M} \rightarrow \Omega_X^1 \otimes \operatorname{gr}_{k+1}^F \mathcal{M} \rightarrow \cdots \rightarrow \Omega_X^g \otimes \operatorname{gr}_{k+g}^F \mathcal{M}\right], $$
which is a complex of coherent $\mathscr{O}_X$-modules placed in degrees $-g, \ldots, 0$. They enjoy many good properties: $E_1$-degeneration of Hodge-de Rham spectral sequence, Kodaira vanishing, decomposition theorem \cite{Saito1988}, and behaves well under duality \cite[Lemma 8.4]{Saitovanishing}. We also require that our mixed Hodge modules are defined over $\Q$, i.e., there's a constructible complex of $\Q$-vector spaces $P$, such that $P\otimes_{\Q}\C \simeq \operatorname{DR}(\M)$.

\begin{definition}[Dolbeault complex of mixed Hodge modules twisted by a one-form]
    Given a mixed Hodge module on a smooth variety $X$ of dimension $g$, $L\in \operatorname{Pic}^0(X)$ and $\theta\in H^0(X,\Omega_X^1)$. We define the following complex of quasi-coherent sheaves on $X$, called the Dolbeault complex of $(\M,L,\theta)$:
    $$\operatorname{Dol}(\M, L,\theta)\defeq \left(\operatorname{gr}_{\bullet}^F \mathcal{M} \otimes L\right) \rightarrow \left(\Omega_X^1 \otimes \operatorname{gr}_{\bullet}^F \mathcal{M} \otimes L\right) \rightarrow \cdots \rightarrow \left(\Omega_X^g \otimes \operatorname{gr}_{\bullet}^F \mathcal{M} \otimes L\right)$$
    starting from degree $-g$ to $0$, with differentials given by the formula $$\omega \otimes m \mapsto(-1)^g\left(\theta \wedge \omega\right) \otimes m+(-1)^g \sum_{j=1}^g\left(d z_j \wedge \omega\right) \otimes \partial_j m.$$

    For a complex of mixed Hodge modules, we take the double complex of the above complex for each individual mixed Hodge module.
    
\end{definition}

When $L=\strucsheaf_X$ and $\theta=0$, this is just the direct sum of graded de Rham complexes. If $\M=\strucsheaf_X$, then 
$\operatorname{Dol}(\strucsheaf_X,L,\theta) = [L\rightarrow \Omega_X^1\otimes L\rightarrow \cdots\rightarrow \Omega_X^g\otimes L]$, 
and the differentials are $(-1)^g\theta$. This kind of complex appeared in \cite{GreenLazarsfeldI}. By $E_1$-degeneration of the Hodge-de Rham complex for mixed Hodge modules \cite{Saito1988}, the cohomology of $\operatorname{Dol}(\M,\strucsheaf_X,0)$ has the same dimension as the cohomology of the de Rham complex of $\M$. The motivation for this definition of Dolbeault complex comes from the formula of harmonic bundles \cite{Simpson1992} and \cite[Theorem 13.1]{holonomicDmodules}. Let $A$ be an abelian variety and $A^\natural$ be the moduli space of line bundles with a flat connection on $A$. $\hat{A}$ is the dual abelian variety and $V=H^0(A,\Omega_A^1)$.

\begin{definition} \label{def:support loci definition}
    Given a mixed Hodge module on $A$, we define
    $$S^k_m(\M)\defeq \{(L,\nabla)\in A^\natural: \dim H^k\left(A, \mathrm{DR}\left(\mathcal{M} \otimes_{\mathscr{O}_A}(L, \nabla)\right)\right )\geq m \},$$
    $$S^k_m(\operatorname{DR}(\M))\defeq \{\chi \in \operatorname{Char}(A): \dim H^k(A, \mathrm{DR}(\mathcal{M}) \otimes \chi )\geq m \},$$
    $$S^k_m(\gr^F\M) \defeq \{(L,\theta)\in \hat{A}\times V: H^k(A,\operatorname{Dol}(\M,L,\theta))  \geq m\},$$
    and let $S^k(\M)\defeq S^k_1(\M)$, $S^k(\gr^F\M)\defeq S^k_1(\gr^F\M)$, and similarly for objects in $\operatorname{D}^b_{\MHM}(A)$.
\end{definition}

It is clear that $\operatorname{RH}:A^\natural \rightarrow \operatorname{Char}(A)$ identifies $S^k_m(\M)$ with $S^k_m(\operatorname{DR}(\M))$ via the Riemann-Hilbert correspondence.

Let $\operatorname{NAH}:A^\natural\rightarrow \hat{A} \times V$ be the map given by non-abelian Hodge correspondence (which is only an $C^\infty$-isomorphism). A \textit{linear subvariety} in $A^\natural$ (resp. $\hat{A}\times V$) is a subvariety of the form $\IM( B^\natural\rightarrow A^\natural)$ (resp. $\IM (\hat{B}\times H^0(B,\Omega_B^1) \rightarrow \hat{A}\times H^0(A,\Omega_A^1))$), where $f:A\rightarrow B$ is a map of abelian varieties. We also say a Zariski closed subset of $A^\natural$ or $\hat{A}\times V$ is \textit{linear} if it's a union of translates of linear subvarieties. We have the following structure theorem for cohomology support loci of holonomic $\D_A$-modules.

\begin{theorem}[\cite{holonomicDmodules}]\label{thm:schnelllinearity}
    $S^k_m(\M)$ is linear for any holonomic $\D$-module on $A$.
\end{theorem}

Our main theorem below is a counterpart of the above theorem, in the Dolbeault setting, and the different cohomology support loci are related by the non-abelian Hodge correspondence. A priori, it's not clear how the two cohomology support loci are related, since we take the graded de Rham complex and the differentials are modified by a one-form.
\begin{theorem} \label{thm: main theorem}
    Let $\M^\bullet\in \operatorname{D}_{\MHM}^b (A)$ be a complex of mixed Hodge modules on an abelian variety $A$. For all $k,m\in \Z$, $S^k_m(\gr^F\M^\bullet)$ is linear in $\hat{A}\times V$. Moreover, the image of  $\operatorname{NAH}|_{S^k_m(\M)}: S^k_m(\M^\bullet)\rightarrow \hat{A}\times V$ is equal to $S^k_m(\gr^F\M^\bullet)$.
\end{theorem}

\begin{corollary}[Generic vanishing of Dolbeault cohomology] \label{cor: generic vanishing for Dolbeault cohomology}
    If $\M$ is a mixed Hodge module on $A$, then $\operatorname{codim} S^k(\gr^F\M)\geq |2k|$, and $\fmd_A(\M)$ is an $m$-perverse coherent sheaf on $\hat{A}\times V$ in the sense of Kashiwara.
\end{corollary}
\begin{proof}
    It's a direct corollary of the main theorem and \cite[Theorem 4.1, 4.2]{holonomicDmodules}.
\end{proof}

\begin{corollary}
    For a mixed Hodge module $\M$ on $A$, $H^k(A,\operatorname{Dol}(\M,L,\theta)) =0$ for $k\neq 0$ and general $(L,\theta) \in \hat{A}\times V$. In particular $\chi(A,\operatorname{Dol}(\M,L,\theta)) \geq 0$ for any $(L,\theta)\in \hat{A}\times V$.
\end{corollary}
\begin{proof}
    It follows from the last corollary, and deformation invariance of Euler characteristic. 
\end{proof}

The next corollary is an equality of dimensions of de Rham and Dolbeault cohomologies of twists of mixed Hodge modules. It says for the de Rham complex of a mixed Hodge module, tensored with a rank one local system, its cohomology can be computed by a complex of quasi-coherent sheaves. It can be seen as a partial generalization of Simpson's formula for cohomology of harmonic bundles \cite[Lemma 2.2]{Simpson1992} (see also \cite[Lemma 14.1]{popaschnell}) where the coefficients here are mixed Hodge modules, and the proof we give here does not use harmonic metrics.
\begin{corollary}\label{cor:equality of dimension}
    Let $X$ be a smooth projective variety and $\M$ is a mixed Hodge module on $X$. Then $$\dim H^k(X,\operatorname{DR}(\M\otimes (L,\nabla))) = \dim H^k(X,\operatorname{Dol}(\M,K,\theta)),$$
    for any $k\in \Z$, where $(L,\nabla)$ is a flat line bundle on $X$ and $(K,\theta)$ is its image under the non-abelian Hodge correspondence.
\end{corollary}
\begin{proof}
    We push forward $\M$ by the Albanese map $\operatorname{alb}:X\rightarrow\alb_X$. By \cite{SaitoMHM}, $\operatorname{alb}_+ \M\in \operatorname{D}^b_{\MHM}(\alb_X)$ is a complex of mixed Hodge modules on $\alb_X$. Then it follows from the main theorem.
\end{proof}

\subsection{Fourier-Mukai transform}

The key tool that we will apply is the Fourier-Mukai transform of algebraic $\D$-modules on abelian varieties. We briefly review the classical one defined in \cite{LaumonFouriertransform} and \cite{Rothstein}, and the extended Fourier transform in \cite{holonomicDmodules}. The latter one takes into account of the information of a filtration. 

We briefly introduce Fourier-Mukai transform for $\D$-modules on abelian varieties following \cite[section 10]{holonomicDmodules}. Let $A$ be an abelian variety, and let $P$ denote the Poincar\'e line bundle on $A \times \hat{A}$. The Fourier-Mukai transform takes bounded complexes of algebraic
$\mathscr{D}_A$-modules to bounded complexes of quasi-coherent sheaves on
$A^\natural$. Consider $A^\natural$, the moduli space of line bundles with flat connection on $A$, the pullback $P^\natural = (\operatorname{id}_A \times \pi)^*P$ of the Poincar\'e bundle to the product $A \times A^\natural$ is endowed
with a universal integrable connection
\[
\nabla^\natural \colon P^\natural \to
\Omega^1_{A \times A^\natural / A^\natural} \otimes P^\natural
\]
relative to $A^\natural$. An algebraic left $\mathscr{D}_A$-module $\mathcal{M}$
may be interpreted as a quasi-coherent sheaf of $\mathscr{O}_A$-modules with a flat connection
\[
\nabla \colon \mathcal{M} \to \Omega_A^1 \otimes \mathcal{M};
\]
then
\[
p_1^*\nabla \otimes \operatorname{id} + \operatorname{id} \otimes \nabla^\natural
\]
defines a relative integrable connection on the tensor product
$p_1^*\mathcal{M} \otimes_{\mathscr{O}_{A \times A^\natural}} P^\natural$,
and we denote the resulting algebraic
$\mathscr{D}_{A \times A^\natural / A^\natural}$-module by the symbol
$p_1^*\mathcal{M} \otimes (P^\natural, \nabla^\natural)$. Given a complex of
algebraic $\mathscr{D}_A$-modules $(\mathcal{M}^\bullet, d)$, we define
\[
\operatorname{DR}_{A \times A^\natural / A^\natural}
\bigl(p_1^*\mathcal{M}^\bullet \otimes (P^\natural, \nabla^\natural)\bigr)
\]
as the single complex determined by the double complex
$(D^{\bullet,\bullet}, d_1, d_2)$, whose term in bidegree $(i,j)$ is equal to
\[
D^{i,j}
=
\Omega^{g+i}_{A \times A^\natural / A^\natural}
\otimes p_1^*\mathcal{M}^j \otimes P^\natural,
\]
and whose differentials are given by the formulas
\[
d_1
=
(-1)^g \bigl(p_1^*\nabla_{\mathcal{M}^j} \otimes \operatorname{id}
+ \operatorname{id} \otimes \nabla^\natural\bigr)
\qquad \text{and} \qquad
d_2 = \operatorname{id} \otimes p_1^*d \otimes \operatorname{id}.
\]

\begin{definition}[Fourier-Mukai transform of algebraic $\D_A$-modules]
    We define the \emph{Fourier--Mukai transform} of the complex
$\mathcal{M}^\bullet$ by the formula
\begin{equation*}
\operatorname{FM}_A(\mathcal{M}^\bullet)
=
\mathbf{R}p_{2*}
\operatorname{DR}_{A \times A^\natural / A^\natural}
\bigl(p_1^*\mathcal{M}^\bullet \otimes (P^\natural, \nabla^\natural)\bigr) \in \operatorname{D}^b_{\operatorname{coh}}(\strucsheaf_{A^\natural}).
\end{equation*}
\end{definition}

Because every differential in the relative de Rham complex is
$\mathscr{O}_{A^\natural}$-linear, $\operatorname{FM}_A(\mathcal{M}^\bullet)$
is naturally a complex of quasi-coherent algebraic sheaves on $A^\natural$. The following theorem says this is an equivalence of categories

\begin{theorem}[\cite{LaumonFouriertransform} and \cite{Rothstein}]
    The Fourier-Mukai transform gives rise to an equivalence of categories
$$
\mathrm{FM}_A: \mathrm{D}_{\operatorname{coh}}^b\left(\mathscr{D}_A\right) \rightarrow \mathrm{D}_{\operatorname{coh}}^b\left(\mathcal{O}_{A^{\natural}}\right)
$$
between the bounded derived category of coherent algebraic $\mathscr{D}_A$-modules and the bounded derived category of coherent algebraic sheaves on $A^{\natural}$.
\end{theorem}

Now we define the Fourier-Mukai transform on the Dolbeault moduli space of Higgs bundles. Let $V=H^0(A,\Omega_A^1)$. We have the following diagram:
\[
\begin{tikzcd}
A \times H^0(A, \Omega^1_A) 
&[3em] A \times \hat{A} \times H^0(A, \Omega^1_A) \arrow[l, "p_{13}"'] \arrow[r, "p_{23}"] \arrow[d, "p_{12}"'] 
&[3em] \hat{A} \times H^0(A, \Omega^1_A) \\
& A \times \hat{A} &
\end{tikzcd}
\]

\begin{definition}[Fourier-Mukai transform, on Dolbeault moduli space] \label{def:fmd}
    Let $\M$ be a mixed Hodge module on an abelian variety $A$. Let $\gr^F\M$ be the associated coherent sheaf on $T^* A \simeq A\times V$. We define the Dolbeault version of the Fourier-Mukai transform of $\M$ as
    $$\fmd_A(\M) \defeq \mathbf{R}\left(p_{23}\right)_*\left(p_{12}^* P \otimes p_1^* \omega_A \otimes p_{13}^*(\id_A \times \iota_V)^* \operatorname{gr}^F \mathcal{M}\right)\in \mathrm{D}_{\operatorname{coh}}^b\left(\mathcal{O}_{\hat{A}\times V}\right),$$
    where $P$ is the Poincaré line bundle on $A\times \hat{A}$, and $\iota_V=-\id_V$ is the involution on $V=H^0(A,\Omega_A^1)$. For $\M^\bullet\in \operatorname{D}^b_{\MHM}(A)$, it's defined in a similar way. 

\end{definition}
The involution $\iota$ appears because of a technical issue, see the lemma below. We can make this definition for any complex of filtered $\D$-modules, but we will only study mixed Hodge modules with its Hodge filtration in this paper.

It's not clear from the expression above that how $\fmd_A(\M)$ is related to the Dolbeault cohomologies of $\M$. It is shown below that this complex restricted to each closed point $(L,\theta)\in\hat{A}\times V$ is the Dolbeault cohomology of $\operatorname{Dol}(\M,L,\theta)$. We have the following lemma.

\begin{lemma}[Lemma 13.3 of \cite{holonomicDmodules}]
    Let $\mathscr{F}$ be a finitely generated $\operatorname{Sym} T_A$-module, and denote by $\tilde{\mathscr{F}}$ the associated coherent sheaf on $T^* A=A \times H^0\left(A, \Omega_A^1\right)$, $\theta_A$ is the universal $1$-form on $T^* A$. Then the complex
    $$ p_1^*\left(\mathscr{O}_A \otimes_{\mathscr{O}_A} \mathscr{F}\right) \rightarrow p_1^*\left(\Omega_A^1 \otimes_{\mathscr{O}_A} \mathscr{F}\right) \rightarrow \cdots \rightarrow p_1^*\left(\Omega_A^g \otimes_{\mathscr{O}_A} \mathscr{F}\right) $$
    where the maps are given by 
    $$ \omega \otimes m \mapsto (-1)^g (\theta_A \wedge \omega)\otimes m + (-1)^g \sum_{j=1}^{g} (dz_j \wedge \omega)\otimes \partial_j m.$$
    is a resolution of $p_1^* \Omega_A^g \otimes(\mathrm{id} \times \iota)^* \tilde{\mathscr{F}}$, where $\iota=-\id$ is the involution of $H^0\left(A, \Omega_A^1\right)$, 
\end{lemma}
Recall that $\fmd_A(\M) \defeq \mathbf{R}\left(p_{23}\right)_*\left(p_{12}^* P \otimes p_1^* \omega_A \otimes p_{13}^*(\id_A \times \iota_V)^* \operatorname{gr}^F \mathcal{M}\right)$, as a consequence of the lemma and base change, 
    if we restrict $\fmd_A(\M)$ to a point $(L,\theta)\in \hat{A}\times V$ which represents a Higgs line bundle with Chern class zero, we get an object that can be represented by the cohomology of the Dolbeault complex $\operatorname{Dol}(\M,L,\theta)$: 
    $$\left[\left(\operatorname{gr}_{\bullet}^F \mathcal{M}\otimes L\right) \rightarrow \left(\Omega_X^1 \otimes_{\mathscr{O}_X} \operatorname{gr}_{\bullet}^F \mathcal{M}\otimes L\right) \rightarrow \cdots \rightarrow \left(\Omega_X^n \otimes_{\mathscr{O}_X} \operatorname{gr}_{\bullet}^F \mathcal{M}\otimes L\right)\right][g].$$
As a corollary to our main result \theoremref{thm: main theorem}, we have the following reconstruction property of the Fourier transform on the Dolbeault moduli space.
\begin{corollary}\label{cor: reconstruction theorem}
Let $\mathcal{M}$ be a simple Hodge module on $A$,
and suppose that the support of $F = \fmd_A(\mathcal{M})$
is equal to $\hat{A}\times V$.
Then $\mathcal{H}^0 F$ is a reflexive coherent sheaf,
locally free on the complement of a finite union of linear
subvarieties of codimension $\geq 4$; and $F$ is uniquely
determined by the restriction of $\mathcal{H}^0 F$
to that open subset. More precisely, $\fmd_A(\M)$ can be reconstructed, up to isomorphism, by applying the functor
$$\tau_{\leq \ell-1}
\circ \mathbf{R}\mathcal{H}om(-,\mathcal{O})
\circ \cdots \circ \tau_{\leq 2}
\circ \mathbf{R}\mathcal{H}om(-,\mathcal{O})
\circ \tau_{\leq 1}
\circ \mathbf{R}\mathcal{H}om(-,\mathcal{O})
\circ j_*$$
to the locally free sheaf $j^*\mathcal{H}^0\mathrm{FM}_A(\mathcal{M})$ where $j:U\hookrightarrow \hat{A}\times V$ is the inclusion of the maximal open set where $\mathcal{H}^0F$ is locally free;
here $\ell$ is any odd integer with $\ell>\dim A$.

\end{corollary}
\begin{proof}
    With our main theorem, the same proof as in \cite[Proposition 22.1]{holonomicDmodules} gives the corollary in our setting.
\end{proof}

The formula above looks like the one for intersection complex of a local system, but in a very different setting. It will be very interesting to have a theory of intersection complex of coherent sheaves to explain this similarity \cite[section 6]{holonomicDmodules}.

Now we give an idea about the extended Fourier-Mukai transform by \cite{holonomicDmodules}, state the result, and leave the precise definition to section 6. It should be viewed as a complex of sheaves interpolating the two Fourier-Mukai transforms that we have defined above.

On the dual abelian variety $\hat{A}=\operatorname{Pic}{ }^0(A)$, there is a canonical extension of vector bundles
$$
0 \rightarrow \hat{A} \times H^0\left(A, \Omega_A^1\right) \rightarrow E(A) \rightarrow \hat{A} \times \mathbb{C} \rightarrow 0
$$
whose extension class in
$$
\begin{aligned}
\operatorname{Ext}^1\left(\mathscr{O}_{\hat{A}}, \mathscr{O}_{\hat{A}} \times H^0\left(A, \Omega_A^1\right)\right) & \simeq H^1\left(\hat{A}, \mathscr{O}_{\hat{A}}\right) \otimes H^0\left(A, \Omega_A^1\right) \\
& \simeq H^0\left(A, \mathscr{T}_A\right) \otimes H^0\left(A, \Omega_A^1\right)
\end{aligned}
$$
is represented by $\sum_j \partial_j \otimes d z_j$. Then $A^{\natural}$ is isomorphic to the preimage of $\hat{A} \times\{1\}$ inside of $E(A)$, and the projection
$$
\pi: A^{\natural} \rightarrow \hat{A}, \quad(L, \nabla) \mapsto L,
$$
is a torsor for the trivial bundle $\hat{A} \times H^0\left(A, \Omega_A^1\right)$.

\begin{definition}
    Let $X$ be a complex manifold, and $\lambda: X \rightarrow \mathbb{C}$ a holomorphic function. A $\lambda$-connection, on a sheaf of $\mathscr{O}_X$-modules $\mathscr{E}$ is a $\mathbb{C}$-linear morphism of sheaves
$$
\nabla: \mathscr{E} \rightarrow \Omega_X^1 \otimes_{\mathscr{O}_X} \mathscr{E}
$$
that satisfies the Leibniz rule with parameter $\lambda$, i.e.,
$$
\nabla(f \cdot s)=f \cdot \nabla s+d f \otimes \lambda s
$$
for local sections $f \in \Gamma\left(U, \mathscr{O}_X\right)$ and $s \in \Gamma(U, \mathscr{E})$. A $\lambda$-connection is called integrable if its curvature $\nabla \circ \nabla: \mathscr{E} \rightarrow \Omega_X^2 \otimes_{\mathscr{O}_X} \mathscr{E}$ is equal to zero.
\end{definition}

\begin{example}
    A $1$-connection is a connection in the usual sense; a $0$-connection is a Higgs sheaf.
\end{example}

Let $\lambda:E(A)\rightarrow \hat{A}\times \C\rightarrow \C=\mathbb{A}^1$ be the composition. Then $\hat{A}\times V=\lambda^{-1}(0)$, and $A^\natural=\lambda^{-1}(1)$. We have defined Fourier-Mukai transforms of mixed Hodge modules on both spaces, and there's a way \cite[section 12]{holonomicDmodules} to define an extended Fourier-Mukai transform $\tfm(\M)$ on $E(A)$, whose restriction to each fiber are the previous transforms.

\begin{theorem} \label{thm: flatness}
    The cohomology sheaf $\mathcal{H}^k(\tfm(\M^\bullet))$ is flat over $\mathbb{A}^1_{\lambda}$ for any $\M^\bullet\in \operatorname{D}^b_{\MHM}(A)$ and any $k\in \Z$.
\end{theorem}
That is to say, for mixed Hodge modules equipped with its Hodge filtration, the extended Fourier-Mukai transform behaves nicely over the twistor line $\mathbb{A}^1_\lambda$. This theorem will follow from \theoremref{thm: main theorem}, with some extra work and constructions inspired by \cite{holonomicDmodules}.

\textbf{Outline of the paper.} In section 2, we collect and review some results of mixed Hodge modules. In section 3, we establish basic functoriality of Fourier transform on Dolbeault moduli space. In section 4, we show that the cohomology support loci of mixed Hodge modules are unitary translates of linear subvarieties. We prove the main result \theoremref{thm: main theorem} in section 5. Section 6 and 7 are devoted to the study of the extended Fourier transform. We apply the results in section 8 to give other realizations of the Tannaka group of Hodge modules following \cite{microlocalI}, in particular, Tannaka group of a Hodge module is the minimal reductive structural group of its projectivized total Fourier transform on $\proj(E(A))$. In section 9, we give a simple proof of the propagation property of cohomology support loci needed in our proofs.

\section*{Acknowledgment}

I thank my advisor Christian Schnell for his constant support. This project can't be completed without discussions with him, especially for suggesting the statement of \theoremref{thm:rational translates} and suggesting working on $\proj(E(A))$. His earlier works are also crucial inputs for this paper.

No mathematical statements and proofs in this paper are involved with AI tools.

\section{Review of some results about mixed Hodge modules}
In this section we collect and review some results of mixed Hodge modules that we are going to use in this paper. The first one is about the dual of Hodge modules. Recall that, by the work of Saito, the associated graded of the Verdier dual $\M'=\mathbf{D}_A(\M)$
is given by the following:
\begin{theorem}[see also Theorem 6.4 of \cite{popaschnell}]
    $$\mathbf{R} \mathcal{H}om{ }_{\mathcal{A}_X^{\bullet}}\left(\operatorname{gr}_{\bullet}^F \mathcal{M}, \mathcal{A}_X^{\bullet}[\dim X]\right) \simeq \omega_X \otimes_{\strucsheaf_X} \operatorname{gr}_{\bullet+2 \dim X}^F \mathcal{M}^{\prime},$$
where sections of $\operatorname{Sym}^k T_X$ act with an extra factor of $(-1)^k$ on the right-hand side. If we consider both sides as coherent sheaves on $T^* X$, we obtain
$$
\mathbf{R} \mathcal{H}om\left(\gr^F\M, \strucsheaf_{T^* X}[\dim X]\right) \simeq p^* \omega_X \otimes(-1)_{T^* X}^* \gr^F\M',
$$
where $p: T^* X \rightarrow X$ is the projection. In particular, $\gr^F\M$ is always a Cohen-Macaulay sheaf of dimension equals $\dim X$.
\end{theorem}

The second result concerns the Hodge filtration on direct images of Hodge modules. When $\M$ underlies a pure Hodge module, by the decomposition theorem of Saito \cite{Saito1988}, $f_+ \M \simeq \bigoplus_i \mathcal{H}^i(f_+ \M)[-i]$, and each cohomology module is also a pure Hodge module with the induced Hodge filtration. Since taking $\gr^F_\bullet$ is an exact functor for Hodge modules and direct images are strict, we know $$\gr^F_\bullet(f_+ \M) \simeq  \bigoplus_i \gr^F_{\bullet}\mathcal{H}^i(f_+ \M)[-i] \in \mathrm{D}^b_{\operatorname{coh}}(\A_X).$$ 

\begin{theorem}[Laumon \cite{LaumonsformulafilteredDmod}, see also Theorem 6.5 of \cite{popaschnell}]
    Let $f: X \rightarrow Y$ be a projective morphism between smooth complex algebraic varieties, and let $\M \in \operatorname{MHM}(X)$. Then with notation as above,
    $$ \operatorname{gr}_{\bullet}^F f_{+}(\mathcal{M}, F) \simeq \mathbf{R} f_*\left(\omega_{X / Y} \otimes_{\strucsheaf_X} \operatorname{gr}_{\bullet+\operatorname{dim} X-\operatorname{dim} Y}^F \mathcal{M} \otimes_{\mathcal{A}_X} f^* \mathcal{A}_Y\right) .$$
\end{theorem}
By Laumon's formula, 
$$\operatorname{gr}_{\bullet}^F\left(\mathcal{H}^i f_{+} \mathcal{M}\right) \simeq R^i f_*\left(\omega_{X/Y}\otimes_{\strucsheaf_X} \operatorname{gr}_{\bullet+\dim X-\dim Y}^F \mathcal{M} \stackrel{\mathbf{L}}{\otimes}_{\A_X} f^* \A_Y\right).$$


For Hodge modules, their graded de Rham complexes satisfy also the decomposition theorem. The first isomorphism in the next theorem is because of strictness of direct images, the second isomorphism in the theorem is due to the decomposition theorem of Saito, see also \cite{MHMproject} for the decomposition theorem of pure complex Hodge modules.
\begin{theorem}[Direct image of graded de Rham complexes \cite{Saito1988}]
    Let $f: X \rightarrow Y$ be a projective morphism between complex manifolds. If $(\mathcal{M}, F_{\bullet} \mathcal{M})$ underlies a pure Hodge module on $X$, one has
$$
\mathbf{R} f_* \operatorname{gr}_p^F \operatorname{DR}(\mathcal{M}) \simeq \operatorname{gr}_p^F \operatorname{DR}\left(f_{+} \mathcal{M}\right) \simeq \bigoplus_{i \in \mathbb{Z}} \operatorname{gr}_p^F \operatorname{DR}\left(\mathcal{H}^i f_{+} \mathcal{M}\right)[-i]
$$
for every $p \in \mathbb{Z}$.

\end{theorem}

\section{Fourier-Mukai transform on the Dolbeault moduli space}

\subsection{Functoriality of the Fourier transform}
Let $\M$ be a mixed Hodge module on an abelian variety $A$ of dimension $g$. Let $V=H^0(A,\Omega_A^1)$ and $\gr^F\M$ be the associated sheaf on $T^*A\simeq A\times V$. We will study the relative Fourier-Mukai transform of $\gr^F \M$ on the Dolbeault moduli space of Higgs line bundles $\hat{A}\times V$, which corresponds to looking at Dolbeault cohomologies of the graded de Rham complexes tensored with Chern class zero line bundles and differentials changed by one-forms. Our main theorem in this section is about the functoriality of our Fourier transforms of mixed Hodge modules under maps of abelian varieties. The non-trivial part is how the graded of the Hodge filtration behave under duality and pushforward, which are answered previously by Saito and Laumon. The statement about duality is used for proving the propagation property \propositionref{prop:propagation}.

Recall that $\hat{A}\times V$ is a group under addition. We denote by $t_{(L, 0)}: \hat{A}\times V \rightarrow \hat{A}\times V$ the translation morphism induced by $L$ with zero Higgs field. We have the following diagram:
\[
\begin{tikzcd}
A \times H^0(A, \Omega^1_A) 
&[3em] A \times \hat{A} \times H^0(A, \Omega^1_A) \arrow[l, "p_{13}"'] \arrow[r, "p_{23}"] \arrow[d, "p_{12}"'] 
&[3em] \hat{A} \times H^0(A, \Omega^1_A) \\
& A \times \hat{A} &
\end{tikzcd}
\]
Then $\fmd_A(\M)$ is defined by the following: $$\fmd_A(\M) \defeq \mathbf{R}\left(p_{23}\right)_*\left(p_{12}^* P \otimes p_1^* \omega_A \otimes p_{13}^*(\id_A \times \iota_V)^* \operatorname{gr}^F \mathcal{M}\right)\in \mathrm{D}_{\operatorname{coh}}^b\left(\mathcal{O}_{\hat{A}\times V}\right),$$
    where $P$ is the Poincaré line bundle on $A\times \hat{A}$, and $\iota_V=-\id_V$ is the involution on $V$.

\begin{theorem}[Functoriality of Fourier transforms] \label{thm:functoriality}
    For mixed Hodge modules on abelian varieties, the Fourier-Mukai transform on the Dolbeault moduli space satisfies the following properties:
    \begin{itemize}
        \item[(a)] For $(L, 0) \in \hat{A}\times V$, one has a canonical and functorial isomorphism
        $$ \fmd_A\left(-\otimes(L, 0)\right)=\mathbf{L}\left(t_{(L, 0)}\right)^* \circ \fmd_A . $$

        \item[(b)] One has a canonical and functorial isomorphism
        $$ \fmd_A \circ \mathbf{D}_A=\left\langle-1_{\hat{A}\times V}\right\rangle^* \mathbf{R} \mathcal{H}om\left(\fmd_A(-), \strucsheaf_{\hat{A}\times V}\right) . $$

        \item[(c)] For a morphism $f: A \rightarrow B$ of abelian varieties, denote by $f^{\natural}: \hat{B}\times H^0(B,\Omega_B^1) \rightarrow \hat{A}\times H^0(A,\Omega_A^1)$ the induced morphism. Then one has canonical and functorial isomorphisms
        $$ \begin{aligned} \mathbf{L}\left(f^{\natural}\right)^* \circ \fmd_A & =\fmd_B \circ f_{+}, \\
        \mathbf{R} f_*^{\natural} \circ \fmd_B & =\fmd_A \circ f^{+}, \text{ when $f$ is smooth.}
        \end{aligned} $$

        \item[(d)] If $\M$ is a pure Hodge module, then $$\fmd_A(\gr^F_\bullet(f_+ \M)) \simeq  \bigoplus_i \fmd_A(\gr^F_{\bullet}\mathcal{H}^i(f_+ \M))[-i] \in \mathrm{D}^b_{\operatorname{coh}}(\A_X),$$
        where $\A_X=\operatorname{Sym}^\bullet T_X$.

    \end{itemize}
\end{theorem}
\begin{proof}
    Part $(a)$: this works in the same way as the usual Fourier-Mukai transform for coherent sheaves, since $t_L^* P\simeq P\otimes p_1^* L$. Note that $(L,0)$ corresponds to a unitary local system under non-abelian Hodge correspondence, therefore $\gr^F\M\otimes L$ is the associated graded of the complex Hodge module $(\M\otimes L,F\otimes L)$ \cite{MHMproject}. 

    Part $(b)$: \begin{align*}
        \fmd_A(\gr^F\M') & =\mathbf{R}(p_{23})_* (p_{12}^* P\otimes p_1^* \omega_A \otimes p_{13}^*(\id\times \iota_V)^*\gr^F\M')\\
        & \simeq \mathbf{R}(p_{23})_* \bigg(p_{12}^* P\otimes p_1^* \omega_A \otimes p_{13}^*(\id\times \iota_V)^* (\id\times \iota_V)^* \\
        & \quad \mathbf{R} \mathcal{H}om{ }\left(\operatorname{gr}^F \mathcal{M}, \strucsheaf_{A\times V}[g]\right) \otimes p_1^{-1} \omega_A^{-1} \bigg)  \\
        & \simeq \mathbf{R}(p_{23})_* \bigg(  (\id_A\times \id_{\hat{A}} \times \iota_V)^* (\id_{A}\times \iota_{\hat{A}}\times \id_{V} )^*\\
        &  \mathbf{R} \mathcal{H}om{ }\left( p_{13}^*(\id\times \iota_V)^*\operatorname{gr}^F \mathcal{M} \otimes p_1^* \omega_A, p_{12}^* P^{-1}\otimes p_1^* \omega_A[g]\right)  \bigg) \\
         \simeq \mathbf{R}(p_{23})_* \bigg(  (\id_A\times \iota_{\hat{A}\times V})^* & \mathbf{R} \mathcal{H}om{ }\left( p_{13}^*(\id\times \iota_V)^*\operatorname{gr}^F \mathcal{M}\otimes p_1^* \Omega_A^g, p_{12}^* P^{-1}\otimes p_1^*\omega_A[g]\right)  \bigg) \\
        & \simeq \iota_{\hat{A}\times V}^* \mathbf{R}(p_{23})_* \bigg(    \mathbf{R} \mathcal{H}om{ }\left( p_{13}^*(\id\times \iota_V)^*\operatorname{gr}^F \mathcal{M}\otimes p_1^* \Omega_A^g, p_{12}^* P^{-1}\right)  \bigg) \\
        & \simeq \iota_{\hat{A}\times V}^* \mathbf{R} \mathcal{H}om\left(\fmd_A(\M), \mathscr{O}_{\hat{A}\times V}\right)
    \end{align*}
    where the first isomorphism is because of Saito's formula above and Grothendieck duality, second isomorphism is because $(\id,\iota_{\hat{A}})^* P\simeq P^{-1}$ on $A\times \hat{A}$, second last isomorphism is again from Grothendieck duality.  

    Part $(c)$: We first prove $\mathbf{L}\left(f^{\natural}\right)^* \circ \fmd_A =\fmd_B \circ f_{+}.$ To simplify the notation, we'll omit $\mathbf{R}$ and $\mathbf{L}$ in the formulas for derived functors. Let $W=H^0(B,\Omega_B^1)$. By Laumon's formula, 
    $$\gr^Ff_+(\M) \simeq (f\times\id_W)_* ( (df)^* \gr^F\M\otimes p_1^* \omega_A\otimes (f\times \id_W)^* p_1^* \omega_B^{-1}),$$
    where 
    \[ \begin{tikzcd}
    A \times V & A \times W \arrow[l, "df"'] \arrow[r] \arrow[d, "f \times \mathrm{id}_W"'] & A \arrow[d, "f"] \\
    &   B \times W \arrow[r] & B.
    \end{tikzcd} \]
    Therefore 
    \begin{align} 
        \fmd_B(f_+\M) & = p_{23*} (p_{12}^*P_B\otimes p_1^* \omega_B\otimes p_{13}^*(\id\times\iota_W)^* \gr^F f_+\M) \nonumber\\
        & \simeq p_{23*} (p_{12}^*P_B\otimes p_1^* \omega_B\otimes p_{13}^*(\id\times\iota_W)^* (f\times\id_W)_* ( df^* \gr^F\M\otimes p_1^* \omega_A\otimes  p_1^* \omega_B^{-1}) ) \nonumber\\
        & \simeq p_{23*} (p_{12}^*P_B \otimes p_{13}^*(\id\times\iota_W)^* (f\times\id_W)_* ( df^* \gr^F\M\otimes p_1^* \omega_A) ).  \label{eq:compatibility}
    \end{align}
    Now consider $(f^\natural)^* \fmd_A(\M)$: 
    \begin{align*}
       (f^\natural)^* \fmd_A(\M) & = (f^\natural)^* p_{23*}(p_{12}^* P_A\otimes p_1^* \omega_A\otimes p_{13}^*(\id\times \iota_V)^* \gr^F\M)\\
        &\simeq p_{23*}h^* (p_{12}^* P_A\otimes p_1^* \omega_A\otimes p_{13}^*(\id\times \iota_V)^* \gr^F\M)\\
        & \simeq p_{23*} (f\times \id_{\hat{B}\times W})_* h^* (p_{12}^* P_A\otimes p_1^* \omega_A\otimes p_{13}^*(\id\times \iota_V)^* \gr^F\M) \\
        & \simeq p_{23*} (f\times \id_{\hat{B}\times W})_* \bigg((f\times\id_{\hat{B}\times W})^* P_B\otimes p_1^*\omega_A\otimes p_{13}^*(\id\times \iota_W)^* (df)^* \gr^F\M  \bigg) \\
        & \simeq p_{23*}\bigg( p_{12}^* P_B\otimes (f\times \id_{\hat{B}\times W})_*  \big(p_1^*\omega_A\otimes p_{13}^*(\id\times \iota_W)^* (df)^* \gr^F\M \big)  \bigg) \\
        & \simeq p_{23*}\bigg( p_{12}^* P_B\otimes (f\times \id_{\hat{B}\times W})_*  p_{13}^*(\id\times \iota_W)^*\big(p_1^*\omega_A\otimes  (df)^* \gr^F\M \big)  \bigg) \\
        & \simeq p_{23*}\bigg( p_{12}^* P_B\otimes p_{13}^*(\id\times \iota_W)^* (f\times \id_{ W})_* \big(p_1^*\omega_A\otimes  (df)^* \gr^F\M \big)  \bigg)\\
        & \simeq \fmd_B(f_+ \M).
    \end{align*}
    where the first isomorphism comes from base change on the right hand side of the diagram \ref{diagram1}; second isomorphism is by chasing the left hand side of \ref{diagram1}; the third isomorphism is because of the functoriality of the Poincaré bundles; the fourth one is because of projection formula; the second last one is by base change on the diagram \ref{diagram2}; the last isomorphism is due to equation \eqref{eq:compatibility}.

    \begin{equation} \label{diagram1}
          \begin{tikzcd}
    B \times \hat{B} \times W \arrow[d, "p_{23}"'] 
      & A \times \hat{B} \times W \arrow[l, "f \times \mathrm{id}_{\hat{B}\times W}"'] 
      \arrow[r, "h"] 
      \arrow[d, "p_{23}"'] 
      & A \times \hat{A} \times V \arrow[d, "p_{23}"] \\
    \hat{B} \times W \arrow[r, "\id"'] 
    & \hat{B} \times W \arrow[r, "f^\natural"] 
      & \hat{A} \times V
    \end{tikzcd} 
    \end{equation}

\begin{equation} \label{diagram2}
    \begin{tikzcd} 
	{B\times \hat{B}\times W} && {A\times \hat{B}\times W} \\
	{B\times W} && {A\times W}
	\arrow["{(\id\times\iota_W)\circ p_{13}}"', from=1-1, to=2-1]
	\arrow["{f\times\id_{B\times W}}"', from=1-3, to=1-1]
	\arrow["{(\id\times\iota_W)\circ p_{13}}"', from=1-3, to=2-3]
	\arrow["{f\times \id_W}"', from=2-3, to=2-1]
\end{tikzcd}
\end{equation}

To prove the second part $\mathbf{R} f_*^{\natural} \circ \fmd_B  =\fmd_A \circ f^{+}$. When $f$ is smooth, $f^+(\M,F) = (f^*\M[\dim A-\dim B],f^* F)$ for smooth pullbacks, hence $$\gr^F_\bullet f^+\M \simeq f^* \gr^F_\bullet \M[\dim A-\dim B].$$ Then the formula follows from the formula for classical Fourier-Mukai transforms for coherent sheaves on abelian varieties.

Part $(d)$: it follows from the decomposition theorem for pure Hodge modules by Saito and strictness of direct images of mixed Hodge modules.
    
\end{proof}

\subsection{Künneth formula for Dolbeault cohomologies}
Let $m: A \times \cdots \times A \rightarrow A$ be the addition morphism. The $d$-fold exterior product $\mathcal{N}_d\defeq \M \boxtimes \cdots \boxtimes \M$ is a mixed Hodge module on $A \times \cdots \times A$, and $\gr^F_\bullet \mathcal{N}_d \simeq \gr^F_{\bullet}\M \boxtimes \cdots \boxtimes \gr^F_{\bullet}\M$. Setting
$$ \M_d=m_+\mathcal{N}_d=m_+(M \boxtimes \cdots \boxtimes M) \in \mathrm{D}^b \operatorname{MHM}(A) ,$$
now we want to find a Künneth formula for Dolbeault cohomology of $\M_d$.
We have the following proposition:
\begin{proposition} On $(\hat{A}\times V)\times \cdots \times (\hat{A}\times V)$, we have the following isomorphism in the bounded derived category of coherent sheaves,
    $$\fmd_{A^d}(\M\boxtimes\cdots\boxtimes M) \simeq \fmd_A(\M) \boxtimes \cdots \boxtimes \fmd_A(\M).$$
\end{proposition}
\begin{proof}
    This is easily proved using base change and functoriality of Poincaré bundles on products of abelian varieties.
\end{proof}
As a consequence of this proposition, we get a Künneth formula for Dolbeault cohomology of exterior products of mixed Hodge modules on abelian varieties. Recall $\operatorname{Dol}(\M,L,\theta)$ is the Dolbeault complex of $(\M,L,\theta)$.
\begin{corollary} \label{cor:kunnethDolbeault}
    We have a Künneth formula for the Dolbeault cohomology of $(\mathcal{N}_d,L_1\boxtimes\cdots\boxtimes L_d,p_1^*\theta+\cdots+p_d^* \theta_d))$:
    \begin{align*}
        \quad \quad & \mathbb{H}^k  (A^d,\operatorname{Dol}  (\mathcal{N}_d,L_1\boxtimes\cdots\boxtimes L_d,p_1^*\theta+\cdots+p_d^* \theta_d))  \simeq \\ 
        & \bigoplus_{a_1+\cdots+a_d=k}\mathbb{H}^{a_1}(A,\operatorname{Dol}(\M,L_1,\theta_1))\otimes \cdots \otimes \mathbb{H}^{a_d}(A,\operatorname{Dol}(\M,L_d,\theta_d))
    \end{align*}
\end{corollary}

\section{Supports of transforms are unitary translates of subtori} 
In \cite{holonomicDmodules}, Schnell proved that the cohomology support loci of the Fourier transform of a holonomic $\D$-modules on an abelian variety $A$ are all translates of linear subvarieties of $A^\natural$, where $A^\natural$ is the space of line bundles on $A$ with an integrable connection. In \cite{torsionpoints}, he showed if a mixed Hodge module has an integral structure, then its cohomology support loci are torsion translates of linear subvarieties. In this section, we follow exactly the same idea as in \cite{torsionpoints} and make it clear that, in the case of mixed Hodge modules, one of the results in the paper \cite{torsionpoints} is essentially the following:
\begin{theorem} \label{thm:unitarytranslates}
    The cohomology support loci of $\M^\bullet\in \operatorname{D}^b_{\MHM}(A)$ on $A^\natural$ are translates of linear subvarieties by points representing rank one unitary connections.
\end{theorem}

We quickly sketch the proof in our setting below, which is essentially the same as in \cite{torsionpoints}. We need the following lemma:

\begin{lemma}[Lemma 1.16 of \cite{torsionpoints}]
    Let $f: A \rightarrow B$ be a homomorphism from $A$ to a lower dimensional abelian variety $B$. Then every intersection
    $$ S_m^k(A, \M) \cap \operatorname{Im} f^\natural( B^\natural) $$ is a finite union of unitary translates of linear subvarieties.
\end{lemma}
\begin{proof}
    By induction, and it's the same as the proof in \cite{torsionpoints}.
\end{proof}

\begin{proof}[Proof of \theoremref{thm:unitarytranslates}]
    The theorem is equivalent to showing that every irreducible component of $S^k_m(\M)$ contains a unitary point. We do induction on dimensions of abelian varieties. For isolated points in cohomology support loci, the same construction in \cite[1.3.2]{torsionpoints} work literally the same, by replacing the word ''torsion points" there by unitary points here. This proves the theorem also for simple abelian varieties.

    If $Z$ is a positive dimensional irreducible component of $S^k_m(\M)$, by \cite{holonomicDmodules}, we know it's a linear subvariety. Then we apply the above lemma above to a map $f:A\rightarrow B$ where $B$ is a complementary abelian variety to $Z$, then the lemma and induction hypothesis imply $Z$ contains a unitary point.
\end{proof}

With this theorem, if in addition $\M$ has a $\Z$-structure, then by Kronecker's theorem on algebraic integers, it follows that these unitary points must have finite order. This is one of the results in \cite{torsionpoints}.

Next we prove a variant of \theoremref{thm:unitarytranslates}. The reason is that we will need to take tensor product of a mixed Hodge module with a unitary local system, but a unitary local system is far from being defined over $\Q$. To stay in the category of mixed Hodge modules defined over $\Q$, we need to take tensor product with a  $\Q$-VHS. Given a rank one $\bar{\Q}$-local system, we can get a $\Q$-local system by taking the direct sum of Galois conjugates. If all Galois conjugates of a rank one $\bar{\Q}$-local system are unitary, then the direct sum is a unitary $\Q$-VHS. The next theorem ensures the existence of such rank one $\bar{\Q}$-local systems.

By a similar definition as in \definitionref{def:support loci definition}, given a constructible complex $P$, we can define its cohomology support loci in $\operatorname{Char}(A)$, and it's denoted by $S_m^k(P)$.

\begin{theorem}\label{thm:rational translates}
    Let $\M^\bullet\in \operatorname{D}^b_{\MHM}(A)$ be a complex of mixed Hodge modules on $A$. For every irreducible component $Z$ of $S_m^k(\operatorname{DR}(\M^\bullet))\subset \operatorname{Char}(A)$, $Z$ contains a closed point $\chi$ defined over $\bar{\Q}$, such that all Galois conjugates of $\chi$ are unitary rank one local systems.
\end{theorem}
\begin{proof}
    The character variety $\operatorname{Char}(A)$ is defined over $\Q$, so it has a natural Galois action by $\operatorname{Gal}(\C/\Q)$. For any map $f:A\rightarrow B$ between abelian varieties, the induced map $f^\natural:\operatorname{Char}(B)\rightarrow \operatorname{Char}(A)$ is a map defined over $\Q$ (by abuse of notation we denote the map by the same symbol as its de Rham counterpart). Because $\operatorname{DR}(\M^\bullet)$ is a $\Q$-constructible complex, its cohomology support loci are defined over $\Q$. 
    
    We do induction on dimensions of abelian varieties. For simple abelian varieties, we need to show an isolated point has all of its Galois conjugates being unitary. Let $Z\subset S_m^k(\operatorname{DR}(\M^\bullet))\subset \operatorname{Char}(A)$ be an isolated point, so it's defined over $\bar{\Q}$. Then all of its Galois conjugates are obviously isolated points of $ S_m^k(\operatorname{DR}(\M^\bullet))\subset \operatorname{Char}(A)$. By \theoremref{thm:unitarytranslates}, they all represent unitary local systems. This argument proves the theorem if $Z$ is an isolated point in cohomology support loci.
    
    Now let $Z\subset S_m^k(\operatorname{DR}(\M^\bullet))\subset \operatorname{Char}(A)$ be an irreducible component of positive dimension. By \theoremref{thm:unitarytranslates}, $Z$ is a unitary translate of a linear subvariety in $\operatorname{Char}(A)$. Then there is a smaller dimensional abelian variety $B$ and a map $f:A\rightarrow B$ between abelian varieties, such that $\dim f^\natural(\operatorname{Char}(B)) +\dim Z=\dim \operatorname{Char}(B)$, therefore $f^\natural(\operatorname{Char}(B))\cap Z$ is non-empty. We have the following by projection formula (see \cite[Lemma 1.16]{torsionpoints}):
    $$(f^\natural)^{-1} S_m^k(\operatorname{DR}\M^\bullet) = S_m^k(\operatorname{DR}f_{+}\M^\bullet).$$
    By induction, $S_m^k(\operatorname{DR}f_{+}\M^\bullet)$ contains a point $\chi$ defined over $\bar{\Q}$, such that all of its Galois conjufares are unitary. Then we pull it back to $A$ and get a local system, still denoted by $\chi\in Z$, and all of its Galois conjugates are unitary.

\end{proof}
Take a direct sum of all the Galois conjugates of such $\chi$, we get a unitary local system defined over $\Q$ denoted by $V_{\chi}$, and we will view it as a $\Q$-VHS.

\section{Linearity of cohomology support loci on Dolbeault moduli space}
In this section, we prove the cohomology support loci of $\operatorname{FM}_A(\M)\subset A^\natural$ map to the cohomology support loci of $\fmd_A(\M)$ on $\hat{A}\times V$, under the non-abelian Hodge correspondence. More generally, we prove the same statement for any $\M^\bullet\in \operatorname{D}^b_{\MHM}(A)$. In particular, the cohomology support loci of $\fmd_A(\M)$ are linear subvarieties translated by points of the form $(L,0)\in \hat{A}\times V$. As an application, we prove on a smooth projective variety $X$, the dimensions of the cohomology of a rank one local system tensored with a perverse sheaf underlying a mixed Hodge module on $X$ can be computed by the cohomology of its Dolbeault complex. 

Recall the definition of cohomology support loci in \definitionref{def:support loci definition}. By base change, equivalently we can define the cohomology support loci $S^k_m$ for de Rham and Dolbeault cohomology as the cohomology support loci of their corresponding Fourier-Mukai transforms. We will treat them as the same thing.

\begin{claim}
    To show the cohomology support loci on Dolbeault and de Rham moduli spaces are the same under the identification given by $\operatorname{NAH}:A^\natural\rightarrow \hat{A}\times V$, it suffices to show the Dolbeault cohomology support loci are linear.
\end{claim}
\begin{proof}
    By definition of $\fmd_A(\M)$, its cohomology support loci is invariant under the $\C^*$-action on the vertical direction of $\hat{A}\times V$. If linearity holds true on $\hat{A}\times V$, then $\C^*$-invariance forces the linear subvarieties to be translated by points of the form $(L,0)\in \hat{A}\times V$, and these points correspond to unitary local systems under non-abelian Hodge correspondence. We know the de Rham cohomology support loci are unitary translates of linear subvarieties. By $E_1$-degeneration of the Hodge-de Rham spectral sequence for unitary local systems tensored with $\M^\bullet$. So linearity of $S^k_m(\gr^F\M)$ and $\C^*$ invariance of it force $S^k_m(\gr^F\M)$ to be translates by points of the form $(L,0)$, and $E_1$-degeneration of Hodge-de Rham spectral sequence plus unitarity of $S^k_m(\M)$ implies linear subvarieties appearing on the Dolbeault moduli space must be identified with the linear subvarieties on the de Rham moduli space, under the non-abelian Hodge correspondence.

\end{proof}

Now the question is to prove linearity for cohomology support loci on $\hat{A}\times V$. In \cite{holonomicDmodules} this is proved on $A^\natural$ for holonomic $\D$-modules, where it used the fact that the bi-algebraic subsets in $A^\natural$ and $\operatorname{Char}(A)$ are exactly translates of linear subvarieties. The idea here is different, because a priori Dolbeault cohomology is not directly related to de Rham cohomology. Our main idea is to do induction on dimensions of abelian varieties, and linearity for pure objects is enough to imply linearity for arbitrary objects in the derived category:
\begin{itemize}
    \item[(1)] first we prove linearity for pure Hodge modules; we use the convolution product that appeared in \cite{kramergenericvanishing} and \cite{torsionpoints}, structure theorem of simple holonomic $\D$-modules \cite{holonomicDmodules}, and the decomposition theorem for pure Hodge modules to further decompose the cohomology support loci of a pure Hodge module into smaller pieces. Then we get some pure Hodge modules that are essentially pullback from abelian varieties with smaller dimension, then we can apply induction.

    \item[(2)] then we prove linearity for mixed objects, by induction on the length of weight filtration, while assuming for $\operatorname{D}^b_{\MHM}$ of quotient abelian varieties, linearity is true; we produce proper linear subvarieties containing the cohomology support loci, and use \lemmaref{lem:induction lemma} for induction;

    \item[(3)] finally we prove linearity for objects in $\operatorname{D}^b_{\MHM}(A)$, the key is still to produce unitary translates of proper linear subvarieties containing the cohomology support loci, and use \lemmaref{lem:induction lemma} to do induction. 
\end{itemize}

The following lemma is the Dolbeault version of \cite[Lemma 1.16]{torsionpoints}. The stability and strictness of $\operatorname{D}^b_{\MHM}$ under pushforward is important.

\begin{lemma} \label{lem:induction lemma}
    Let $\M\in \operatorname{D}^b_{\MHM}(A)$, $f: A \rightarrow B$ be a homomorphism from $A$ to a lower dimensional abelian variety $B$, and $f^\natural: B\times H^0(B,\Omega_B^1)\rightarrow A\times H^0(A,\Omega_A^1)$. Then
    $$(f^\natural)^{-1} S_m^k(A, \gr^F\M)=S_m^k\left(B, \gr^F f_+ \M\right)$$
\end{lemma}
\begin{proof}
    The same as the proof of \cite[Lemma 1.16]{torsionpoints}.
\end{proof}

The significance of this lemma is, once we know if a component of $S_m^k(A,\gr^F\M)$ is contained in a proper linear subvariety, then we can conclude it is already linear, by  this lemma and induction on dimensions of abelian varieties. Therefore the main idea of our proof is to produce proper linear subvarieties (up to unitary translates) containing cohomology support loci.

We first recall a structure theorem of simple Hodge modules, which follows from \cite[Theorem 21.1]{holonomicDmodules}, see also \cite[Theorem 9.1]{HodgemodulesonTori}. In the following, $\chi(\M)$ means the Euler characteristic of the perverse sheaf $\operatorname{DR}(\M)$.

\begin{theorem} \label{thm:chi=0}
    Let $\mathcal{M}$ be a simple complex Hodge module on an abelian variety $A$. If $\chi(\M)=0$, then there is an abelian variety $B$ of dimension less than $\dim A$, a surjective morphism $f: A \rightarrow B$ with connected fibers, and a simple Hodge module $\mathcal{N}$ on $B$, such that $ \mathcal{M} \otimes_{\mathscr{O}_A}(L, \nabla_u) \simeq f^* \mathcal{N},$
    for a suitable unitary connection $(L, \nabla_u) \in A^{\natural}$. Moreover, we have $\operatorname{Supp} \mathcal{H}^0\left(\mathrm{FM}_B(\mathcal{N})\right)=B^{\natural}$ and $$ \operatorname{Supp} \operatorname{FM}_A(\mathcal{M})=(L, \nabla_u) \otimes \operatorname{im}\left(f^{\natural}: B^{\natural} \rightarrow A^{\natural}\right) . $$
\end{theorem}
This theorem follows from the proof of \cite[Theorem 21.1]{holonomicDmodules} when the simple holonomic $\D$-module underlies a Hodge module, plus the fact that the category of complex Hodge modules is semi-simple, and a simple complex Hodge module is also a simple holonomic $\D$-module. It tells us that Hodge modules with zero Euler characteristic are essentially pullbacks from lower dimensional abelian varieties up to simple direct summands and unitary translations. Hence we can do induction on dimensions of abelian varieties, using functoriality of the Fourier transform \theoremref{thm:functoriality} and \lemmaref{lem:induction lemma}.

Now we prove some easy properties of $S^k_m(\gr^F\M)$.

\begin{proposition} \label{prop:C-star invariance}
    Let $\M$ be a mixed Hodge module on $A$. The cohomology support loci $S^k_m(\gr^F\M)$ satisfy the following properties.
    \begin{itemize}
        \item[(a)] They are invariant under the $\C^*$-action on $V$ and every irreducible component of $S^k_m(\gr^F\M)$ intersects $\hat{A}\times \{0\}$ in a non-empty set;
        \item[(b)] The intersection $S^k_m(\gr^F\M) \cap (\hat{A}\times \{0\})$ corresponds to the unitary points in $S^k_m(\M)$ under the non-abelian Hodge correspondence, and are unions of translates of subtori of $\hat{A}\times\{0\}$;
        \item[(c)] For $k\neq 0$, $S^k_m(\gr^F\M)$ are proper subvarieties of $\hat{A}\times V$.
    \end{itemize} 
\end{proposition}   
\begin{proof}
    Part $(a)$: the $\C^*$-invariance is because $\fmd_A(\M)$ is the Fourier transform of the graded module $\gr^F_\bullet \M$. By $\C^*$-invariance, each irreducible component of $S^k_m(\gr^F\M)$ must have a non-empty intersection with $\hat{A}\times\{0\}$.

    Part $(b)$: $\hat{A}\times \{0\}\subseteq \hat{A}\times V$ corresponds to unitary connections in $A^\natural$ under non-abelian Hodge correspondence \cite{Simpson1992}. The Dolbeault complex $\operatorname{Dol}(\M,L,0)$ is the same as the direct sum of graded de Rham complexes when the Higgs field is zero. By $E_1$-degeneration of the Hodge-de Rham spectral sequence for complex Hodge modules, a unitary point in $S^k_m(\M)$ if and only if its image under the correspondence is inside $S^k_m(\gr^F\M) \cap (\hat{A}\times \{0\})$. Being translates of subtori is a direct consequence of the above discussion and \theoremref{thm:schnelllinearity}.
    
    Part $(c)$: We also know $S^k_m(\M)$ are proper subvarieties by \cite[Proposition 19.2]{holonomicDmodules} for $k\neq 0$, combining with the part $(a)$ and $(b)$, we know $S^k_m(\gr^F\M)$ on the Dolbeault moduli space must be proper subvarieties when $k\neq 0$.

\end{proof}

\begin{remark}
    In the remaining proofs of this paper, we will frequently take take tensor product $\M^\bullet\otimes (L_\chi,\nabla)$ of a complex of mixed Hodge modules with a unitary rank one local system $(L_{\chi},\nabla)$, and we will call it a unitary translate of the complex of mixed Hodge modules. However, since we only work with $\Q$-mixed Hodge modules, to stay in this category, we can only take tensor product with $\Q$-VHS. Therefore to replace unitary translates, we use \theoremref{thm:rational translates}. By choosing an appropriate unitary connection $(L_{\chi},\nabla)$, there's a unitary $\Q$-VHS $(V_\chi,\nabla)$ containing $(L_{\chi},\nabla)$ as a direct summand. Since all of our statements hold true under taking direct summands, adding extra summands in the proof doesn't affect our conclusion. We will also call $\M^\bullet \otimes (V_\chi,\nabla)$ a unitary translate of $\M^\bullet$, although $\M^\bullet \otimes (L_\chi,\nabla)$ is only a direct summand of it. 
\end{remark}

We first prove equality of cohomology support loci under Simpson correspondence, for pure complex Hodge modules. Since we work in the complex setting here, by a unitary translate we just mean tensoring with a unitary rank one flat connection. 

\begin{theorem}[Pure case of \theoremref{thm: main theorem}] \label{thm:pure case}
    Let $\M$ be a pure Hodge module on an abelian variety $A$ and for all $k,m\in \Z$, $S^k_m(\gr^F\M)$ is linear in $\hat{A}\times V$, and are translates by points of the form $(L,0)\in \hat{A}\times V$. Moreover, the image of  $\operatorname{NAH}|_{S^k_m(\M)}: S^k_m(\M)\rightarrow \hat{A}\times V$ is equal to $S^k_m(\gr^F\M)$.
\end{theorem}
\begin{proof}
    The cohomology support loci of $S^k_m(\M)$ are unitary translates of linear subvarieties is proved in \theoremref{thm:unitarytranslates}. A unitary connections $(L,\nabla_u)$ is mapped to the same line bundle with zero Higgs field $(L,0)$ by the non-abelian Hodge correspondence \cite{Simpson1992}. As observed before, it suffices to prove $S^k_m(\gr^F\M)$ are linear. We prove this by induction on the dimensions of abelian varieties. This is obviously true for zero dimensional abelian varieties. Given a Hodge module $\M$ on an abelian variety $A$, now we suppose for lower dimensional abelian varieties and pure Hodge modules on them, the theorem is already proved. 

    Since the category of pure Hodge modules is semi-simple, without loss of generality we can assume $\M$ is a simple Hodge module. If $\chi(\M)=0$, then $\M$ is isomorphic to the pullback of a simple Hodge module from a lower dimensional abelian variety tensored with a unitary connection by \theoremref{thm:chi=0}. Then the result follows from induction and the functoriality results in \theoremref{thm:functoriality} part $(a)$ and part $(c)$. Now let $\chi(\M)\neq 0$. 

    First consider the case $k=0$. Because $\chi(\M)\neq 0$ and $\M$ is simple, by the result of \cite{holonomicDmodules}, $\mathcal{H}^0(\fm_A(\M))$ is torsion free and is supported on the whole $A^\natural$. Because $\fmd_A(\gr^F\M)$ is a deformation of $\fm_A(\M)$ along the twistor space (see \cite[section 12]{holonomicDmodules}), by semi-continuity $\mathcal{H}^0(\fmd_A(\M))$ is then supported on the whole $\hat{A}\times V$. By \propositionref{prop:C-star invariance}, we know $S^k(\gr^F\M)$ and $\mathcal{H}^k(\fmd_A(\M))$ are also supported on proper subvarieties of $\hat{A}\times V$ for $k\neq 0$. From base change we have the following, 
    \begin{align*}
        \bigcup_{k\geq 0} S^k(\M) & = \bigcup_{k\geq 0} \operatorname{Supp}(\mathcal{H}^k\fm_A(\M)) = A^\natural, \\
        \bigcup_{k\geq 0} S^k(\gr^F\M) & = \bigcup_{k\geq 0} \operatorname{Supp}(\mathcal{H}^k\fmd_A(\M)) = \hat{A}\times V.
    \end{align*} 
    These imply $S^0(\M)=\operatorname{Supp}\mathcal{H}^0\fm_A(\M) =A^\natural$ and $S^0(\gr^F\M)=\operatorname{Supp}\mathcal{H}^0\fmd_A(\M)=\hat{A}\times V$ must both equal to the whole parameter spaces. This proves that the support of $S^0(\gr^F\M)$ is identified with the support of $S^0(\M)$ under the non-abelian Hodge correspondence.

    Now consider the case when $k\neq 0$. We assume $k>0$, and the case $k<0$ can be treated analogously. Let $m:A^d=A\times \cdots\times A\rightarrow A$ be the multiplication map, and let $\M_d\defeq m_+(\M\boxtimes\cdots\boxtimes\M)$ which is a direct sum of shifts of Hodge modules by the decomposition theorem. It follows from the Künneth formula for Dolbeault complexes \corollaryref{cor:kunnethDolbeault} that
    \begin{align*}
        \bigcup_{a_1+\cdots+a_d=k d}\left(\bigcap S^{a_i}(\gr^F\M)\right) & =S^{k d}\left(\gr^F\M_d\right).
    \end{align*}
    \begin{lemma} \label{lem:Mdlemma}
        $S^k(\gr^F\M)=S^{kd}(\gr^F\M_d)$.
    \end{lemma}
    \begin{proof}
        By the propagation property (see \cite{mellinpropagation} and \corollaryref{cor:propagationDolbeault}), we know $$\emptyset=S^{g+1}\subseteq S^g\subseteq \cdots\subseteq S^1\subseteq S^0\supseteq S^{-1}\supseteq \cdots \supseteq S^{-g}\supseteq S^{-g-1}=\emptyset$$ since $\operatorname{Dol}(\M,L,\theta)$ is a complex of quasi-coherent algebraic sheaves concentrated in negative degrees $-g,\cdots,0$. Because $a_1+\cdots a_d=kd$, then one of the $a_i$ is greater than or equal to $k$, hence $\bigcap_{a_1+\cdots+a_d=kd} S^{a_i}\subset S^k$. $S^k\subseteq \bigcup_{a_1+\cdots+a_d=k d}\left(\bigcap S^{a_i}(\M)\right) $ because we can let all $a_i$ to be equal to $k$. Therefore, we have proved that $S^k(\gr^F\M)=S^{kd}(\gr^F\M_d)$. 
    \end{proof}

    For $d\gg 0$, we have the following: 
    \begin{equation}\label{eq:decompose support loci}
         S^{kd}(\gr^F\M_d) = \bigcup_{q\geq 1} S^{kd-q} (\gr^ F\mathcal{H}^q(\M_d)).
    \end{equation}
    It's because of the decomposition theorem for $\gr^F\M_d$ by \theoremref{thm:functoriality} part $(d)$, and the fact that the Dolbeault complexes only have cohomologies in degrees $-g,\cdots,g$. When $k>0$ and $d$ is very large, we can take $q\geq 1$.

    \begin{lemma}
        The cohomology Hodge modules $\mathcal{H}^q(\M_d)$ for $q\neq 0$ are direct sums of pullbacks of Hodge modules from smooth maps to lower dimensional abelian varieties, up to unitary translates. 
    \end{lemma}
    \begin{proof}
        By \cite[Lemma 1.20]{torsionpoints}, the cohomology support loci of $\mathcal{H}^q(\M_d)$ are proper linear subvarieties, for $q\neq 0$. Then $\chi(\mathcal{H}^q(\M_d)) =0$, and each simple summand of $\mathcal{H}^q(\M_d)$ must be pullback Hodge modules from lower dimensional abelian varieties composed with a unitary translation, by \theoremref{thm:chi=0}.
    \end{proof}
    
    With this lemma, using functoriality \theoremref{thm:functoriality}, \lemmaref{lem:induction lemma}, by induction, and  \eqref{eq:decompose support loci}, we know $S^{kd}(\gr^F\M_d)$ is linear (for $k\neq 0$), therefore $S^k(\gr^F\M)$ is also linear by \lemmaref{lem:Mdlemma}.

    Now consider the general cohomology support loci $S^k_m(\gr^F\M)$ where $m\geq 2$. Consider the case $k\neq 0$ first. Since $k\neq 0$, we know they are proper subvarieties of $\hat{A}\times V$ by \propositionref{prop:C-star invariance}. If $S^k(\gr^F \M) = S^k_m(\gr^F\M)$, then we are already done. If not, then $S^k_m(\gr^F \M)\subsetneq S^k(\gr^F \M)$ is contained in a union of unitary translates of proper linear subvarieties of $\hat{A}\times V$, by what we have proved before about $S^k(\gr^F\M)$. Say an irreducible component of $S^k_m(\gr^F\M)$ is contained in a proper linear subvariety $(L,0)+\hat{B}\times H^0(B,\Omega_B^1)$. By \theoremref{thm:functoriality} part $(a)$, we can translate $\fmd_A(\M)$ by a point of the form $(L,0)$, so that we make the this translate of linear subvariety centered at the origin. Now this irreducible component of $S^k_m(\gr^F\M)$ is contained in $ \operatorname{Image}(\hat{B}\times H^0(B,\Omega_B^1)\rightarrow \hat{A}\times H^0(A,\Omega_A^1))$, we can apply \lemmaref{lem:induction lemma} and induction hypothesis on $B$ to conclude that this component of $S^k_m(\gr^F\M)$ is a unitary translate of a linear subvariety. 

    Consider the case $k=0$. If $S^0_m(\gr^F\M)=S^0(\gr^F\M)=S^0_{\chi(\M)}(\gr^F\M)$, we are done. If they are not equal, we have $m>\chi(\M)$. By the deformation invariance of Euler characteristic of the Dolbeault cohomology, we know that if at a point $(L,\theta)\in \hat{A}\times V$, the Dolbeault cohomology has $\dim H^0\geq m>\chi(\M)$, then one of its cohomologies $H^k$ will also have non-zero dimension for some $k\neq 0$, which means $S^0_m(\gr^F\M)\subseteq \cup_{k\neq 0} S^k(\gr^F\M)$. Now we proceed as in the paragraph above using induction and \lemmaref{lem:induction lemma}. The theorem is proved.
    
\end{proof}

We finish the proof of \theoremref{thm: main theorem}. Note that here and in the rest of this paper we work with derived category of $\Q$-mixed Hodge modules, since a complex version of it does not exist yet. By a unitary translate, we mean from now in the rest of this paer by taking tensor product with $V_\chi = \bigoplus_{g\in \operatorname{Gal}(\C/\Q)} L_{g\cdot \chi}$, where $L_\chi$ is a rank one flat connection given by \theoremref{thm:rational translates}. The key idea is to find proper linear subvarieties containing the cohomology support loci, and use \lemmaref{lem:induction lemma} to conclude that it's actually linear.

\begin{proposition}
    For $\M^\bullet\in \operatorname{D}_{\MHM}^b (A)$ on an abelian variety $A$ and for all $k,m\in \Z$, $S^k_m(\gr^F\M^\bullet)$ is linear in $\hat{A}\times V$. Moreover, the image of  $\operatorname{NAH}|_{S^k_m(\M)}: S^k_m(\M^\bullet)\rightarrow \hat{A}\times V$ is equal to $S^k_m(\gr^F\M^\bullet)$.
\end{proposition}
\begin{proof}
    For a single mixed Hodge module $\M$, we use induction on the weight filtration. Assume we have proved this proposition for complexes of mixed Hodge modules on smaller dimensional abelian varieties, and for mixed Hodge modules with length of the weight filtration less than $d$.
    
    Suppose $W_{-1}\M=0$, $\gr^W_d\M\neq 0$ and $\gr^W_l=0$ for $l>d$. For $k\neq 0$, we know $S^k_m(\gr^F\M)\subset S^k(\gr^FW_{d-1}\M)\cup S^k(\gr^F\gr^W_d \M)$. By induction and what we know about pure Hodge modules, $S^k(\gr^FW_{d-1}\M)\cup S^k(\gr^F\gr^W_d \M)$ is a union of unitary translates of proper linear subvarieties, and we take unitary translates of them, apply \lemmaref{lem:induction lemma} and induction on the dimension of abelian varieties to conclude $S^k_m(\gr^F\M)$ is linear. Here we used the fact that $\operatorname{D}^b_{\MHM}$ is stable under pushforward.

    For $S^0(\gr^F\M)$, if $S^0(\gr^FW_{d-1}\M)\cup S^0(\gr^F\gr^W_d \M)$ is a union of proper linear subvarieties, then we are done as in the last paragraph. If not, then consider the following exact sequence
    \begin{align*}
        H^0(\operatorname{Dol}(W_{d-1}\M,L,\theta))) & \rightarrow H^0(\operatorname{Dol}(\M,L,\theta)))\rightarrow \\
        & \rightarrow H^0(\operatorname{Dol}(\gr^W_d\M,L,\theta)))\rightarrow H^1(\operatorname{Dol}(W_{d-1}\M,L,\theta))).
    \end{align*}
    Suppose either $S^0(\gr^FW_{d-1}\M)$ or $S^0(\gr^F\gr^W_d \M)$ has full support. Because $W_{d-1}\M$ has no Dolbeault cohomology at $H^{-1}$ and $H^1$ at a generic $(L,\theta)$ by \propositionref{prop:C-star invariance}, then at this $(L,\theta)$, $H^0(\operatorname{Dol}(\M,L,\theta)))\neq 0$. In this case we conclude that $S^0(\gr^F\M)$ has full support. 
    
    For $m\geq 2$ and if $S^0_m(\gr^F\M)\subsetneq S^0_1(\gr^F\M)$, by the long exact sequence of cohomology, a point in $S^0_m(\gr^F\M)$ must also be contained the union of some higher cohomology support loci $S^i(\gr^FW_{d-1}\M)\cup S^j(\gr^F\gr^W_d\M)\cup S^0_r(\gr^FW_{d-1}\M)\cup S^0_s(\gr^F\gr^W_d\M)$, for $i,j\neq 0$ and 
    $$S^0_r(\gr^FW_{d-1}\M)\subsetneq S^0(\gr^FW_{d-1}\M), \quad S^0_s(\gr^F\gr^W_{d}\M)\subsetneq S^0_s(\gr^F\gr^W_{d}\M).$$ 
    This is a union of unitary translates of proper linear subvarieties, then we can use apply (maybe after some unitary translates of $\M$) \lemmaref{lem:induction lemma}, and conclude $S^0_m(\gr^F\M)$ is linear for any $m\geq 1$.

    Now we consider complexes of mixed Hodge modules. For $\M^\bullet\in \operatorname{D}^b_{\MHM}(A)$, we have a spectral sequence at any $(L,\theta)\in \hat{A}\times V$:
    $$E_2^{p,q}= H^p(\operatorname{Dol}(\mathcal{H}^q\M^\bullet,L,\theta))\Rightarrow H^{p+q}(\operatorname{Dol}(\M^\bullet,L,\theta)).$$
    For $p\neq 0$, we know $S^p(\gr^F\mathcal{H}^q\M^\bullet)$ is a proper linear subvariety. If $S^0(\gr^F\mathcal{H}^{p+q}\M^\bullet) $ is a proper linear subvariety, then we know $S^{p+q}(\gr^F\M^\bullet)$ is contained in a union of proper linear subvarieties, and we can apply \lemmaref{lem:induction lemma} and induction on dimension of abelian varieties to show that $S^{p+q}(\gr^F\M^\bullet)$ is linear. 
    
    Assume in this paragraph that $S^0(\gr^F\mathcal{H}^{p+q}\M^\bullet) $ is equal to $\hat{A}\times V$, then since for $p\neq 0$ and generic $(L,\theta)$, the $p$-th Dolbeault cohomology of a single mixed Hodge module at $(L,\theta)$ is zero. Hence at a generic point $(L,\theta)$, $H^0(\operatorname{Dol}(\mathcal{H}^{p+q}\M^\bullet,L,\theta)) = H^{p+q}(\operatorname{Dol}(\M^\bullet,L,\theta))$. In this case $S^{p+q}(\gr^F\M^\bullet) = S^0(\gr^F\mathcal{H}^{p+q}\M^\bullet) = \hat{A}\times V$. 
    
    If for some $m\geq 2$, $S^{p+q}_m(\gr^F\M^\bullet)\subsetneq S^{p+q}_1(\gr^F\M^\bullet)$, then similar as before, the cohomology support loci $S^{p+q}_m(\gr^F\M^\bullet)$ should be contained in the union of $S^p_{k_p}(\gr^F\mathcal{H}^q\M^\bullet)$ for some appropriate numbers $k_p$. $S^p_{k_p}(\gr^F\mathcal{H}^q\M^\bullet)$ are proper linear subvarieties in this case because of the assumption $S^{p+q}_m(\gr^F\M^\bullet)\subsetneq S^{p+q}_1(\gr^F\M^\bullet)$. This is a union of unitary translates of proper linear subvarieties, again we can apply \lemmaref{lem:induction lemma} (maybe after some unitary translates) and by induction hypothesis, $S^{p+q}_m(\gr^F\M^\bullet)$ is a union of unitary translates of linear subvarieties.

\end{proof}

We have the following corollaries to the main theorem.
\begin{corollary}
    For a mixed Hodge module $\M$ on $A$, 
    $$\dim H^k(A,\operatorname{DR}(\M\otimes (L,\nabla))) = \dim H^k(A,\operatorname{Dol}(\M,K,\theta)),$$
    for any $k\in \Z$ and $(L,\nabla)\in A^\natural$, where $(K,\theta)\in \hat{A}\times V$ is mapped to $(L,\nabla)$ under the non-abelian Hodge correspondence.
\end{corollary}
\begin{proof}
    This is because the cohomology support loci are mapped to each other under the correspondence, therefore the dimensions of cohomologies are the same on both de Rham and Dolbeault sides.
\end{proof}

\begin{corollary}
    Let $X$ be a smooth projective variety and $\M$ is a mixed Hodge module on $X$. Then $$\dim H^k(X,\operatorname{DR}(\M\otimes (L,\nabla))) = \dim H^k(X,\operatorname{Dol}(\M,K,\theta)),$$
    for any $k\in \Z$, where $(L,\nabla)$ is a flat line bundle on $X$ and $(K,\theta)$ is its image under the non-abelian Hodge correspondence.
\end{corollary}

\section{The extended Fourier-Mukai transform}

\subsection{Definition of extended Fourier-Mukai transform}
We describe an extension of the Fourier-Mukai transform to the case of algebraic $\mathscr{R}_A$-modules, following \cite[section 12]{holonomicDmodules}. Recall that $\mathscr{D}_A$ is generated as an $\mathscr{O}_A$-algebra by a basis of vector fields $\partial_1, \ldots, \partial_g \in H^0\left(A, \mathscr{T}_A\right)$, subject to the relations
$$
\left[\partial_i, \partial_j\right]=0 \quad \text { and } \quad\left[\partial_i, f\right]=\partial_i f .
$$
Likewise, $\mathscr{R}_A$ is generated as an $\mathscr{O}_A[z]$-algebra by $z \partial_1, \ldots, z \partial_g$, subject to
$$
\left[z \partial_i, z \partial_j\right]=0 \quad \text { and } \quad\left[z \partial_i, f\right]=z \cdot \partial_i f .
$$
$\mathscr{R}_A$ is isomorphic to $\left(p_1\right)_* \mathscr{R}_{A \times \mathbb{C} / \mathbb{C}}$, where $\mathscr{R}_{A \times \mathbb{C} / \mathbb{C}}$ denotes the $\mathscr{O}_{A \times \mathbb{C}}$-subalgebra of $\mathscr{D}_{A \times \mathbb{C} / \mathbb{C}}$ generated by $z \mathscr{T}_{A \times \mathbb{C} / \mathbb{C}}$. If $\M$ is an algebraic $\mathscr{R}_A$-module, then the associated quasi-coherent sheaf $\widetilde{M}$ on $A \times \mathbb{C}$ is naturally a module over $\mathscr{R}_{A \times \mathbb{C} / \mathbb{C}}$, and vice versa.

Now fix an algebraic $\mathscr{R}_A$-module $M$. Let $\mathscr{R}_{A \times E(A) / E(A)}$ denote the subalgebra of $\mathscr{D}_{A \times E(A) / E(A)}$ generated by $\lambda \mathscr{T}_{A \times E(A) / E(A)}$. Recall $\lambda:E(A)\rightarrow \mathbb{A}^1$ is the map to the twistor line. The following lemma tells there exists a generalized connection for the Poincaré bundle on $A\times E(A)$, therefore we can apply the idea of Fourier-Mukai transform to $\mathscr{R}_A$-modules.

\begin{lemma}[Lemma 11.5 of \cite{holonomicDmodules}]
    Let $\widetilde{P}$ denote the pullback of the Poincaré bundle to $A \times E(A)$. Then there is a canonical generalized relative connection
    $$ \widetilde{\nabla}: \widetilde{P} \rightarrow \Omega_{A \times E(A) / E(A)}^1 \otimes \widetilde{P} $$
    that satisfies the Leibniz rule $\widetilde{\nabla}(f \cdot s)=f \cdot \widetilde{\nabla} s+d_{A \times E(A) / E(A)} f \otimes \lambda s$.
\end{lemma}
The tensor product

$$
(\operatorname{id} \times \lambda)^* \widetilde{M} \otimes_{\mathscr{O}_{A \times E(A)}} \widetilde{P}
$$
naturally has the structure of $\mathscr{R}_{A \times E(A) / E(A)}$-module on $A \times E(A)$: concretely, the module structure is given by $\lambda(m \otimes s)=(\lambda m) \otimes s=m \otimes \lambda s$ and $\lambda \partial_i(m \otimes s)= \left(z \partial_i m\right) \otimes s+m \otimes \widetilde{\nabla}_{\partial_i}(s)$. We may therefore consider the relative de Rham complex
$$
\mathrm{DR}_{A \times E(A) / E(A)}\left((\operatorname{id} \times \lambda)^* \widetilde{M} \otimes_{\mathscr{O}_{A \times E(A)}}(\widetilde{P}, \widetilde{\nabla})\right),
$$
which is defined just as in the case of $\mathscr{D}$-modules.

\begin{definition}
    The Fourier-Mukai transform of an algebraic $\mathscr{R}_A$-module $M$ is
$$
\widetilde{\operatorname{FM}}_A(M)=\mathbf{R}\left(p_2\right)_* \operatorname{DR}_{A \times E(A) / E(A)}\left((\operatorname{id} \times \lambda)^* \widetilde{M} \otimes_{\mathscr{O}_{A \times E(A)}} \widetilde{P}\right) ;
$$
it is an object of $\mathrm{D}^b\left(\strucsheaf_{E(A)}\right)$, the bounded derived category of quasi-coherent algebraic sheaves on $E(A)$.
\end{definition}

We have the following compatibility result.

\begin{proposition}[Proposition 13.1 of \cite{holonomicDmodules}]
Let $\mathcal{M}$ be a coherent algebraic $\mathscr{D}_A$-module with a good
filtration $F_\bullet \mathcal{M}$. Then the extended Fourier--Mukai transform
\[
  \tfm_A(\M)
  \in \mathrm{D}^b_{\mathrm{coh}}\bigl(\mathcal{O}_{E(A)}\bigr)
\]
of the associated graded $\mathscr{R}_A$-module has the following properties:
\begin{enumerate}
  \item[(i)] It is equivariant for the natural $\mathbb{C}^*$-action on the
  vector bundle $E(A)$.

  \item[(ii)] Its restriction to
  $A^\natural = \lambda^{-1}(1)$ is canonically isomorphic to
  $\operatorname{FM}_A(\mathcal{M})$.

  \item[(iii)] Its restriction to
  $\widehat{A} \times H^0(A,\Omega_A^1) = \lambda^{-1}(0)$
  is canonically isomorphic to
  \[
    \fmd_A(\M) =\mathbf{R}(p_{23})_*
    \left(
      p_{12}^*P
      \otimes p_1^*\Omega_A^g
      \otimes p_{13}^*(\operatorname{id}\times\iota)^*
      \operatorname{gr}^F \mathcal{M}
    \right),
  \]
  where $\iota=-\operatorname{id}$ is the obvious involution of
  $H^0(A,\Omega_A^1)$.
\end{enumerate}
\end{proposition}
The extended Fourier-Mukai transform should be thought as a family of Fourier-Mukai transforms over the twistor line $\mathbb{A}_\lambda^1$, interpolating the Fourier transforms on Dolbeault and de Rham moduli spaces. When $\M$ is a mixed Hodge module equipped with its Hodge filtration, it is expected that $\tfm_A(\M)$ behaves well over $\mathbb{A}_\lambda^1$.

\subsection{Functoriality of the extended Fourier-Mukai transform}
In this section, we prove the necessary functoriality of the extended Fourier-Mukai transform for mixed Hodge modules. They are only needed to prove flatness over the twistor line. 

Recall we have the map $\lambda:E(A)\rightarrow \mathbb{A}^1$. Given $L\in \operatorname{Pic}^0(A)$, it has a unique unitary connection $\nabla_L$. There exists a canonical section of $\lambda$ (it exists by extending the canonical action of $\C^*$ on $\mathbb{A}^1_\lambda$), $s_L:\mathbb{A}^1\rightarrow E(A)$ where $s_L(0)= (L,0)\in \hat{A}\times V\subset E(A)$, $s_L(1)= (L,\nabla_L)\in A^\natural\subset E(A)$. We define the following algebraic morphism:
$$t_L\defeq \id_{E(A)}+s_L\circ \lambda:E(A)\rightarrow E(A).$$
This is analogous to the translation morphism $t_{(L,0)}:\hat{A}\times V\rightarrow \hat{A}\times V$ on the Dolbeault moduli space.

We can pushforward, take smooth pullback of filtered $\D$-modules which underlie mixed Hodge modules (for example defined in \cite[chapter 8]{MHMproject}). Given a map $f:A\rightarrow B$ between abelian varieties, we still denote by $f^\natural:E(B)\rightarrow E(A)$ the natural map.

\begin{theorem}[Functoriality of the extended Fourier transforms] \label{thm:extendedfunctoriality}
    For mixed Hodge modules on abelian varieties, the extended Fourier-Mukai transform on $E(A)$ satisfies the following properties:
    \begin{itemize}
        \item[(a)] One has a canonical and functorial isomorphism
        $$ \tfm_A\left(-\otimes(L,\nabla_L ))\right)=\mathbf{L}(t_{L})^* \circ \tfm_A . $$

        \item[(b)] For a morphism $f: A \rightarrow B$ of abelian varieties, denote by $f^{\natural}: \hat{B}\times H^0(B,\Omega_B^1) \rightarrow \hat{A}\times H^0(A,\Omega_A^1)$ the induced morphism. Then one has canonical and functorial isomorphisms
        $$ \begin{aligned} \mathbf{L}\left(f^{\natural}\right)^* \circ \tfm_A & =\tfm_B \circ f_{+}, \\
        \mathbf{R} f_*^{\natural} \circ \tfm_B & =\tfm_A \circ f^{+}, \text{ when $f$ is smooth.}
        \end{aligned} $$

    \end{itemize}
\end{theorem}

\begin{proof}
    We write derived functors using the usual notations for the corresponding sheaf-theoretic operations.
    
    Part $(a)$: denote by $\widetilde{L}$ the associated quasi-coherent sheaf on $A\times \mathbb{A}^1$, viewed as a module over $\mathscr{R}_{A\times \mathbb{A}^1/\mathbb{A}^1}$. Then $\widetilde{L}$ on $A\times \mathbb{A}^1$ also has a structure of relative $\lambda$-connection, where by abuse of notations, $\lambda:A\times \mathbb{A}^1\rightarrow \mathbb{A}^1$ is the second projection. Consider $\id_A\times \lambda:A\times E(A)\rightarrow A\times \mathbb{A}^1$, then $(\id_A\times \lambda)^* \widetilde{L}$ also has a relative $\lambda$-connection on $A\times E(A)$, i.e.,
$$ \widetilde{\nabla}_L: \widetilde{L} \rightarrow \Omega_{A \times E(A) / E(A)}^1 \otimes \widetilde{L} $$
that satisfies the Leibniz rule $\widetilde{\nabla}_L(f \cdot s)=f \cdot \widetilde{\nabla}_L s+d_{A \times E(A) / E(A)} f \otimes \lambda s$.

\begin{align*}
t_L^* \circ \widetilde{\operatorname{FM}}_A(\M)
&=
t_L^*
p_{2*}
\operatorname{DR}_{A\times E(A)/E(A)}
\left(
  (\operatorname{id}_A\times\lambda)^*\widetilde{\mathcal{M}}
  \otimes \widetilde{P}_A
\right)
\\
&\simeq
p_{2*}
(\operatorname{id}_A\times t_L)^*
\operatorname{DR}_{A\times E(A)/E(A)}
\left(
  (\operatorname{id}_A\times\lambda)^*\widetilde{\mathcal{M}}
  \otimes \widetilde{P}_A
\right)
\\
&\simeq
p_{2*}
\operatorname{DR}_{A\times E(A)/E(A)}
\left(
  (\operatorname{id}_A\times\lambda)^*\widetilde{\mathcal{M}}
  \otimes
  (\operatorname{id}_A\times t_L)^*\widetilde{P}_A
\right)
\\
&\simeq
p_{2*}
\operatorname{DR}_{A\times E(A)/E(A)}
\left(
  (\operatorname{id}_A\times\lambda)^*\widetilde{\mathcal{M}}
  \otimes
  \widetilde{P}_A
  \otimes
  \widetilde{(L,\nabla_L)}
\right)
\\
&\simeq
p_{2*}
\operatorname{DR}_{A\times E(A)/E(A)}
\left(
  (\operatorname{id}_A\times\lambda)^*
  \widetilde{(\mathcal{M}\otimes (L,\nabla_L))}
  \otimes
  \widetilde{P}_A
\right)
\\
&=
\widetilde{\operatorname{FM}}_A
\left(
  \mathcal{M}\otimes (L,\nabla_L)
\right).
\end{align*}
The first isomorphism is because of base change; the third isomorphism is because $(\id_A\times t_L)^* (\widetilde{P}_A,\widetilde{\nabla}) \simeq (\widetilde{P}_A,\widetilde{\nabla}) \otimes (\widetilde{L},\widetilde{\nabla}_L)$; the last isomorphism is because $(\id_A\times \lambda)^*\widetilde{\M}\otimes \widetilde{(L,\nabla_L)} \simeq (\id_A\times \lambda)^*\widetilde{(\M\otimes (L,\nabla_L))}$, viewed as modules with $\lambda$-connections on $A\times E(A)$.

Part $(b)$: we have the following diagram:
    \begin{equation}\label{diagram4}
    \begin{tikzcd}
	& {E(B)} & {E(A)} \\
	{B\times E(B)} & {A\times E(B)} & {A\times E(A)} \\
	{B\times\mathbb{A}^1} && {A\times\mathbb{A}^1}
	\arrow["{f^\natural}", from=1-2, to=1-3]
	\arrow["{\id_B\times \lambda}", from=2-1, to=3-1]
	\arrow[from=2-2, to=1-2]
	\arrow["{f\times \id_{E(B)}}"', from=2-2, to=2-1]
	\arrow["{\id_{A}\times f^{\natural}}", from=2-2, to=2-3]
	\arrow[from=2-3, to=1-3]
	\arrow["{\id_A \times \lambda}", from=2-3, to=3-3]
	\arrow[from=3-3, to=3-1]
\end{tikzcd}
    \end{equation}

Then it follows that:    
    \begin{align*}
(\mathbf{L}f^\natural)^* \circ \widetilde{\operatorname{FM}}_A(\mathcal{M})
& = (f^\natural)^* p_{2*}\,
\operatorname{DR}_{A\times E(A)/E(A)}
\left(
  (\operatorname{id}_A\times\lambda)^*\widetilde{\mathcal{M}}
  \otimes \widetilde{P}_A
\right)
\\
&\simeq
p_{2*}\,
(\id_A\times f^\natural)^*
\operatorname{DR}_{A\times E(A)/E(A)}
\left(
  (\operatorname{id}_A\times\lambda)^*\widetilde{\mathcal{M}}
  \otimes \widetilde{P}_A
\right)
\\
&\simeq
p_{2*}\,
\operatorname{DR}_{A\times E(B)/E(B)}
\left(
  (\operatorname{id}_A\times\lambda)^*\widetilde{\mathcal{M}}
  \otimes
  (\operatorname{id}_A\times f^\natural)^*\widetilde{P}_A
\right)
\\
&\simeq
p_{2*}\,
\operatorname{DR}_{A\times E(B)/E(B)}
\left(
  (\operatorname{id}_A\times\lambda)^*\widetilde{\mathcal{M}}
  \otimes
  (f\times \operatorname{id}_{E(B)})^*\widetilde{P}_B
\right)
\\
&\simeq
p_{2*}\,
(f\times \operatorname{id}_{E(B)})_*
\operatorname{DR}_{A\times E(B)/E(B)}
\left(
  (\operatorname{id}_A\times\lambda)^*\widetilde{\mathcal{M}}
  \otimes
  (f\times \operatorname{id}_{E(B)})^*\widetilde{P}_B
\right)
\\
&\simeq
p_{2*}\,
\operatorname{DR}_{B\times E(B)/E(B)}
\left(
  (\operatorname{id}_B\times\lambda)^*\widetilde{f_+\mathcal{M}}
  \otimes \widetilde{P}_B
\right)
\\
&=
\tfm_B\circ f_{+}(\M),
\end{align*}
all the $p_2$ here denote the obvious second projection to $E(A)$ or $E(B)$; the first isomorphism is by applying base change to the upper right corner of diagram \ref{diagram4}; the third isomorphism is because of $(\id_A\times f^\natural)^* \widetilde{P}_A \simeq (f\times \id_{E(B)})^* \widetilde{P}_B$; the second last isomorphism is by diagram chasing; the last isomorphism is because of the projection formula, the fact that taking de Rham complex commutes with pushforward, and the fact that pushforward of a complex of mixed Hodge modules is strict with respect to the filtration.

The second statement for part $(b)$ follows from the first one by adjunction, and the fact that $\tfm_A: \operatorname{D}^b(\mathscr{R}_A)\rightarrow \operatorname{D}^b(\strucsheaf_{E(A)})$ is an equivalence of categories \cite[p.17]{holonomicDmodules}:
\begin{align*}
\operatorname{Hom}\bigl(\tfm_A(\M), \mathbf{R}f^\natural_* \tfm_B(\mathcal{N}) \bigr)
&\simeq
\operatorname{Hom}\bigl(\mathbf{L}(f^\natural)^* \tfm_A(\M), \tfm_B(\mathcal{N}) \bigr)\\
&\simeq
\operatorname{Hom}\bigl(\tfm_B(f_+ \M),\tfm_B(\mathcal{N})  \bigr)\\
&\simeq
\operatorname{Hom}\bigl(f_+ \M,\mathcal{N}\bigr)\\
&\simeq
\operatorname{Hom}\bigl(\M,f^{+}\mathcal{N}\bigr)\\
&\simeq
\operatorname{Hom}\bigl(\tfm_A(\M),
\tfm_A(f^{+}\mathcal{N})\bigr).
\end{align*}
Note that here we require $f:A\rightarrow B$ to be proper and smooth. The first isomorphism is by adjunction, since $f^\natural:E(B)\rightarrow E(A)$ is a closed embedding in this case. The fourth isomorphism is because of adjunction and all the fact that the functors are strictly compatible with the Hodge filtration \cite{SaitoMHM}. By Yoneda lemma, we have $\mathbf{R} (f^{\natural})_* \circ \tfm_B  =\tfm_A \circ f^{+}$.

\end{proof}

\section{Flatness of the total Fourier transform}
In this section, we use the results in previous sections to prove the extended Fourier transform behaves nicely over the twistor line $\mathbb{A}^1_\lambda$ for $\M^\bullet \in \operatorname{D}^b_{\MHM}(A)$. 

The main result of this section is, the cohomology sheaves of the total Fourier transform $\widetilde{\fm}(\M^\bullet)$ of $\M^\bullet \in \operatorname{D}^b_{\MHM}$ are flat over the twistor line $\mathbb{A}_{\lambda}^1$. The cohomology sheaves are already flat over $\lambda\neq 0$ by the $\C^*$-action on $\mathbb{A}^1_\lambda$, therefore proving flatness amounts to proving the cohomology sheaves have no $\lambda$-torsion. We use an ad hoc method based on \cite[Proof of Theorem 20.1]{holonomicDmodules}, which says for $\operatorname{D}^b_{\MHM}$, the truncation of the Fourier transform is still a Fourier transform of another object in $\operatorname{D}^b_{\MHM}$. Using the functoriality of $\widetilde{\fm}$ and $\fmd$, we can truncate the Fourier transform simultaneously on $E(A)$ and $\hat{A}\times V$. If the cohomology sheaves are not flat, then $\fmd$ will have extra cohomology sheaves, we can always truncate simultaneously so that for $\lambda\neq 0$ we get zero complex, but the extra cohomology sheaves on $\lambda=0$ will give a contradiction to our main result \theoremref{thm: main theorem}. Unfortunately we couldn't find a more natural proof instead of this ad hoc method.

\begin{proposition}\label{prop:lambdatorsion}
    For a single mixed Hodge module $\M$ on $A$, $\mathcal{H}^0(\widetilde{\fm}(\M))$ does not have $\lambda$-torsion.
\end{proposition}
\begin{proof}
    
    By the Tor-spectral sequence for restricting to $\lambda=0$, we have
    $$E_{2}^{p,q} = \mathcal{H}^p(\mathcal{H}^q\widetilde{\fm}(\M)\otimes^L \strucsheaf_{\hat{A}\times V}) \Rightarrow \mathcal{H}^{p+q}(\widetilde{\fm}(\M) \otimes^L \strucsheaf_{\hat{A}\times V}) = \mathcal{H}^{p+q}(\fmd(\M)).$$
    Because $\strucsheaf_{\hat{A}\times V}$ has a resolution $[\strucsheaf_{E(A)}\xrightarrow{\lambda} \strucsheaf_{E(A)}]$, this spectral sequence only has nonzero terms in $E_2^{0,q}$ and $E_{2}^{-1,q+1}$, and degenerates at $E_2$. We have
    $$\mathcal{H}^{-1}(\fmd(\M)) \simeq \mathcal{H}^{-1}(\mathcal{H}^0\tfm(\M)\otimes^L \strucsheaf_{\hat{A}\times V}).$$
    Since $\mathcal{H}^{-1}(\mathcal{H}^0\tfm(\M)\otimes^L \strucsheaf_{\hat{A}\times V}) \simeq \ker(\cdot \lambda:\mathcal{H}^0\tfm(\M)\rightarrow \mathcal{H}^0\tfm(\M))$, it suffices to show for a single mixed Hodge module, $\mathcal{H}^{-1}(\fmd(\M))=0$. But by the definition of $\fmd$, it is the Fourier-Mukai transform of a single sheaf $\gr^F\M$ on $A\times V$, hence cannot have cohomology sheaves in negative degrees. This implies $\mathcal{H}^0\tfm(\M)$ has no $\lambda$-torsion.

    

\end{proof}

We need the following \theoremref{thm:truncation} in proving flatness.  It says any truncation of the total Fourier transform of a complex of mixed Hodge modules is the transform of another complex of mixed Hodge modules. It originates from Theorem 20.1 of \cite{holonomicDmodules}, whose proof carries over to $\tfm$ of complexes of mixed Hodge modules, using the functoriality for transforms of complexes of mixed Hodge modules which is proved in \theoremref{thm:extendedfunctoriality}. Technically we need to include a proof of \theoremref{thm:truncation} here, but since our proof of \theoremref{thm: flatness} is based on the idea of taking truncations and reduce multiplicity of total Fourier transform, the proof of \theoremref{thm: flatness} and \theoremref{thm:truncation} will be very similar, and it's a kind of repeation if we write it down in details. Therefore we just include here a paragraph to illustrate what has to be changed in adapting Schnell's proof in our setting. 

One thing to note that in adapting the proof of Theorem 20.1 of \cite{holonomicDmodules}, at various points we will need to translate the complex of mixed Hodge modules by unitary points, by tensoring with $V_\chi$ to stay within the derived category of $\Q$-mixed Hodge modules. On the Fourier transform, tensoring with $V_\chi$ just adds direct sums of rank one unitary translates of $\tfm_A(\M^\bullet)$. One has to replace the natural map $\M\rightarrow f^+(\tau_{\geq k} f_+ \M)$ in \cite[Proof of Theorem 20.1]{holonomicDmodules} by the natural map 
$$\M\rightarrow f^+(\tau_{\geq k} (\M\otimes V_{\chi^{-1}})) \otimes V_{\chi}$$
where $\chi$ is given by \theoremref{thm:rational translates}, and then complete to a triangle in $\operatorname{D}^b_{\MHM}(A)$. The reason for doing this is that in the holonomic setting \cite{holonomicDmodules}, one can take arbitrary translate, but in our setting of $\Q$-mixed Hodge modules, the argument of Schnell by reducing multiplicity each component at a time cannot work. The argument in \cite{holonomicDmodules} in fact simutanesously reduces the multiplicity of Fourier transform for all of these copies of unitary translates.  Because by the Galois invariance of support loci for $\Q$-mixed Hodge modules, if we reduce multiplicity along a single component of cohomology support loci, then multiplicities along all Galois conjugates of a component of cohomology support loci are simutaneously reduced. This can be easily seen from \cite[Theorem 15.2]{holonomicDmodules} and \cite[Proof of Proposition 1.13]{torsionpoints}. For a constructible sheaf, one can define a similar Fourier transform (sometimes called the Mellin transform, e.g. \cite{mellinpropagation}) for its cohomology support loci in $\operatorname{Char}(A)$, and this Fourier transform is defined over $\Q$ if the constructible complex is defined over $\Q$ \cite[Proposition 14.8]{holonomicDmodules}. \cite[Theorem 15.2]{holonomicDmodules} says Riemann-Hilbert correspondence identifies these different transforms on $A^\natural$ and $\operatorname{Char}(A)$ analytically. So if we reduce the generic multiplicity of Fourier transform along a single component of its support, then we automatically reduce mutiplicity along all Galois conjugates of this component. Therefore in our setting, introducing twists by $V_{\chi}$ and $V_{\chi^{-1}}$ is necessary, and does not cause a problem in adapting Schnell's proof. Same consideration also happens for the proof of \theoremref{thm: flatness}.

Another interesting part of \theoremref{thm:truncation} is that, it says for $\Q$-mixed Hodge modules, the essential image for their Fourier transforms are closed under taking cohomology sheaves. 

\begin{theorem} \label{thm:truncation}
    Let $F=\tfm_A(\mathcal{M})$ for some $\mathcal{M} \in \mathrm{D}_{\MHM}^b\left(A\right)$. Then for every $n \in \mathbb{Z}$, the truncations $\tau_{\leq n} F$ and $\tau_{\geq n} F$ are again the total Fourier transforms of complexes of mixed Hodge modules.
\end{theorem}


Now we prove~\theoremref{thm: flatness}.

\begin{proof}[Proof of \theoremref{thm: flatness}]
    By \theoremref{thm:truncation}, we assume $\tfm(\M^\bullet) = \mathcal{H}^1(\tfm(\M^\bullet))[-1]$ only has cohomology sheaf at degree $1$, and the goal is to show it has no $\lambda$-torsion.  By the Tor-spectral sequence for restricting to $\lambda=0$, $\mathcal{H}^1(\fmd(\M^\bullet)) \simeq \mathcal{H}^1(\tfm(\M))\otimes \strucsheaf_{\hat{A}\times V}$, $\mathcal{H}^0(\fmd(\M^\bullet))\simeq \ker(\cdot\lambda:\mathcal{H}^1(\tfm(\M^\bullet))\rightarrow \mathcal{H}^1(\tfm(\M^\bullet)))$, and $\fmd(\M^\bullet)$ has no other cohomology sheaves. It remains to show that $\mathcal{H}^0(\fmd(\M^\bullet))$ must be zero. 

    The idea of the proof is to reduce multiplicity of $\tfm(\M^\bullet)$ along each component of its support, and show that $\lambda$-torsion (if there is any) does not change in this process. This is a contradiction because in the end we get into the situation that $\tfm(\M^\bullet)$ is equal to zero but there is non-trivial $\lambda$-torsion, which is absurd.

    \quad
    
    \textit{Step 1:} We first show that we can reduce to the case where $\tfm(\M^\bullet)$ does not have full support. Consider the distinguished triangle $\tau_{\leq 0}\M^\bullet\rightarrow \M^\bullet \rightarrow \tau_{\geq 1} \M^\bullet\rightarrow$ in $\operatorname{D}^b_{\MHM}(A)$ and apply functors $\fmd$ and $\tfm$ to it, we have $\mathcal{H}^0(\fmd(\tau_{\leq 0}\M^\bullet))\simeq \mathcal{H}^0(\fmd(\M^\bullet)) $, and
    $$0\rightarrow \mathcal{H}^1(\tfm(\tau_{\leq 0}\M^\bullet))\rightarrow  \mathcal{H}^1(\tfm(\M^\bullet)) \rightarrow \mathcal{H}^1(\tfm(\tau_{\geq 1}\M^\bullet))\rightarrow .$$
    By the spectral sequence $E_2^{p,q}=\mathcal{H}^p(\tfm(\mathcal{H}^q\M^\bullet))\Rightarrow \mathcal{H}^{p+q}(\tfm(\M^\bullet))$, we know $$\mathcal{H}^1(\tfm(\tau_{\geq 1}\M^\bullet)) \simeq \mathcal{H}^0(\tfm(\mathcal{H}^1\M^\bullet)).$$
    Hence the $\lambda$-torsion of $\mathcal{H}^1(\tfm(\M^\bullet))$, if exists, must come from $\mathcal{H}^1(\tfm(\tau_{\leq 0}\M^\bullet))$. $\mathcal{H}^1(\tfm(\tau_{\leq 0}\M^\bullet))$ does not have full support, because its restriction to $\lambda=1$ is the $\mathcal{H}^1$ of the de Rham transform $\fm (\tau_{\leq 0}\M^\bullet)$, and it's known in \cite[Proposition 19.2]{holonomicDmodules} that $\mathcal{H}^1(\fm (\tau_{\leq 0}\M^\bullet))$ is not of full support. 

    \quad
    
    \textit{Step 2:} We prove there is a way to modify $\M^\bullet$ to get another complex of mixed Hodge modules with zero $\tfm$, but with nonzero $\mathcal{H}^0\fmd$. The key point of this step is that, the sheaf $\mathcal{H}^0(\fmd)$ is not changed throughout the process, even if we get some different complexes of mixed Hodge modules and apply the functor $\fmd$, and it's always equal to the $\lambda$-torsion of $\mathcal{H}^1(\tfm)$ applied to these different complexes. The idea of modifying the complex comes from \cite[proof of Theorem 20.1]{holonomicDmodules}. 
    
    Now we assume $F\defeq \tfm(\M^\bullet) = \mathcal{H}^1(\tfm(\M^\bullet))[-1]$ does not have full support. Then $F$ is supported in unions of ``unitary" translates of subvarieties of the form $t_L^*(f^\natural (E(B)))$ where $f:A\rightarrow B$ is a smooth morphism of abelian varieties, $\dim B<\dim A$, and $L\in \hat{A}$ by \theoremref{thm: main theorem}. By \theoremref{thm:rational translates} and part $(a)$ of \theoremref{thm:extendedfunctoriality} we can assume one irreducible component of $\tfm(\M^\bullet)$ with largest dimension is equal to $Z=t_{L_{\chi^{-1}}}^* \IM(f^\natural:E(B)\rightarrow E(A))$. Assume $\codim Z=2r\geq 2$.  

    Let $i:Z\rightarrow E(A)$. Now consider the pullback $\mathbf{L} i^* F$ to the subvariety $Z$. By construction, the $1$-st cohomology sheaf of $\mathbf{L} i^* F$ is supported on all of $Z$, while all cohomology sheaves in degree bigger than one are zero. By functoriality of the total Fourier transform for pushforward \theoremref{thm:extendedfunctoriality} $(b)$, we have $\mathbf{L} i^* F \simeq \tfm_B\left(f_{+} (\mathcal{M}^\bullet \otimes V_{\chi^{-1}} ) \right)$, and therefore $\mathcal{H}^1 \tfm_B\left(f_{+} (\mathcal{M}^\bullet \otimes V_{\chi^{-1}} )\right) \neq 0$. If we now define
$$
\mathcal{N}=\tau_{\geq 1} f_{+} (\mathcal{M}^\bullet \otimes V_{\chi^{-1}} ) \in \mathrm{D}_{\MHM}^{[1, r]}\left(B\right),
$$
we obtain a distinguished triangle
$$
\tfm_B\left(\tau_{\leq 0} f_{+} (\mathcal{M}^\bullet \otimes V_{\chi^{-1}} )\right) \rightarrow \tfm_B\left(f_{+} (\mathcal{M}^\bullet \otimes V_{\chi^{-1}} )\right) \rightarrow \tfm_B(\mathcal{N}) \rightarrow $$
we conclude that $\mathcal{H}^1 \tfm_B(\mathcal{N})$ and $i^* \mathcal{H}^1 F$ are isomorphic at the generic point of $Z$, and that $\mathcal{H}^i \mathrm{FM}_B(\mathcal{N})$ are torsion sheaves for $i>1$.

The adjunction morphism $\mathcal{M}^\bullet \rightarrow f^{+} (f_{+} (\mathcal{M}^\bullet \otimes V_{\chi^{-1}} ))\otimes V_{\chi}$ induces a morphism $\mathcal{M}^\bullet \rightarrow (f^{+} \mathcal{N})\otimes V_{\chi}$. We choose $\mathcal{M}^{\prime} \in \mathrm{D}_{\MHM}^b\left(A\right)$ so that there is a distinguished triangle in $\operatorname{D}^b_{\MHM}(A)$:
$$
\mathcal{M}^{\prime} \rightarrow \mathcal{M} \rightarrow f^{+} \mathcal{N}\otimes V_{\chi} \rightarrow \mathcal{M}^{\prime}[1].$$
Since $f$ is smooth of relative dimension $r$, we have $f^{+} \mathcal{N} \otimes V_{\chi}\in \mathrm{D}_{\MHM}^{[1-r, 0]}\left(A\right)$ (see \cite[p.33]{HTTDmodules} for the shift in degrees); consequently, $\mathcal{M}^{\prime} \in \mathrm{D}_{\MHM}^{\leq 1}\left(A\right)$ (see also \cite[p.33, proof of Theorem 20.1]{holonomicDmodules}). Let $F^{\prime}\defeq \tfm_A\left(\mathcal{M}^{\prime}\right)$; we claim the following.


\begin{claim}
    \begin{itemize}
        \item[(1)]  $\tau_{\leq 0} F^{\prime} = \tau_{\leq 0} F=0$, and $\supp (\tau_{\leq 1}F') \subseteq \supp(F)$;

        \item[(2)] the multiplicity of $\mathcal{H}^1 F^{\prime}$ along $Z$ is strictly smaller than that of $\mathcal{H}^1 F$.
    \end{itemize}
\end{claim}
\begin{proof}
    The first part is obvious: by functoriality of $\widetilde{\fm}$ transform for smooth pullbacks \theoremref{thm:extendedfunctoriality}, we have $\tfm_A\left(f^{+} \mathcal{N}\right) \simeq \mathbf{R} i_* \tfm_B(\mathcal{N})$, and so the support of the total Fourier transform of $f^{+} \mathcal{N}\otimes V_{\chi}$ is entirely contained in translates of $Z$ by direct summands in $V_\chi$. Moreover, $\tfm_A\left(f^{+} \mathcal{N}\right)$ belongs to $\mathrm{D}_{\text{coh}}^{\geq 1}(\mathcal{O}_{E(A)})$ since $\mathcal{N}=\tau_{\geq 1}f_{+} (\mathcal{M}^\bullet \otimes V_{\chi^{-1}} )\in \operatorname{D}_{\MHM}^{[1,r]}(B)$ and $\tfm_A\left(f^{+} \mathcal{N}\right) \simeq \mathbf{R} i_* \tfm_B(\mathcal{N})$; apply the functor $\tfm$ to the triangle above, we have $\tau_{\leq 0} F^{\prime} \simeq \tau_{\leq 0} F$, they are equal to zero since $F$ by assumption is a single sheaf placed at degree one.

For the statement about multiplicity, we have an exact sequence by applying $\tfm$ to the triangle $
\mathcal{M}^{\prime} \rightarrow \mathcal{M} \rightarrow f^{+} \mathcal{N} \otimes V_{\chi}\rightarrow \mathcal{M}^{\prime}[1]:$
$$
0 \rightarrow \mathcal{H}^1 F^{\prime} \rightarrow \mathcal{H}^1 F \rightarrow \mathcal{H}^1 \tfm_A\left(f^{+} \mathcal{N} \otimes V_\chi\right). $$
The left hand side is zero since $\tfm_A\left(f^{+} \mathcal{N}\right)$ belongs to $\mathrm{D}_{\text{coh}}^{\geq 1}(\mathcal{O}_{E(A)})$. Because $\mathcal{H}^1 \tfm_B(\mathcal{N})$ and $i^* \mathcal{H}^1 F$ are isomorphic at the generic point of $Z$, by functoriality \theoremref{thm:extendedfunctoriality} part $(a)$, we have  
$\mathcal{H}^1 \tfm_A\left(f^{+} \mathcal{N} \otimes V_\chi\right) $ restricted to each Galois conjugate of $Z$ are isomorphic to $\mathcal{H}^1 F$ restricted to that component.
By \cite[Lemma 20.2]{holonomicDmodules}, we know the multiplicity of $\mathcal{H}^1F'$ is less that that of $\mathcal{H}^1F$.
\end{proof}
In the proof, in fact we simutaneously reduced the multiplicity of $F$ along all Galois conjugates of $Z$.

\textit{Step 3: $\lambda$-torsion is not changed.} Now we look at the $\lambda$-torsion part: because $\mathcal{N}$ is a complex of mixed Hodge modules truncated at degree $1$, $\mathcal{H}^1(\tfm_B(\mathcal{N})) = \mathcal{H}^0(\tfm(\mathcal{H}^1\mathcal{N}))$ by a spectral sequence argument, and it has no $\lambda$-torsion by \propositionref{prop:lambdatorsion}. By functoriality of $\tfm$ for smooth pullbacks \theoremref{thm:extendedfunctoriality}, together with induction on dimensions of abelian varieties, $\mathcal{H}^1 \tfm_A\left(f^{+} \mathcal{N}\right) \simeq i_* \mathcal{H}^1(\tfm_B(\mathcal{N}))$ has no $\lambda$-torsion as well. By the exact sequence above, the $\lambda$-torsion of $\mathcal{H}^1F$ comes from $\mathcal{H}^1F'$, and $ \mathcal{H}^0(\fmd(\M^\bullet)) \simeq \mathcal{H}^0(\fmd(\M')) \simeq \ker(\cdot\lambda:\mathcal{H}^1F'\rightarrow \mathcal{H}^1F')$. The $\lambda$-torsion in $\mathcal{H}^1$ is not changed, even if we replace $F$ by $F'$.

Consider $\tau_{\leq 1} F'$, then it's the total Fourier transform of another complex of mixed Hodge modules $\M''\in \operatorname{D}^b_{\MHM}(A)$ by \theoremref{thm:truncation}, with support contained in that of $F$. Moreover by the construction above,
$$\mathcal{H}^0(\fmd(\M'')) \simeq \mathcal{H}^0(\fmd(\M^\bullet)) = \ker(\cdot\lambda:\mathcal{H}^1F'\rightarrow \mathcal{H}^1F') = \ker(\cdot\lambda:\mathcal{H}^1F\rightarrow \mathcal{H}^1F),$$
but multiplicity of $\mathcal{H}^1(\tfm(\M'')) = \mathcal{H}^1 F'$ along $Z$ is less. 


\quad

\textit{Step 4:} By possibly changing $\M^\bullet\in \operatorname{D}^b_{\MHM}$, we have reduced the multiplicities of $\mathcal{H}^1(\tfm(\M^\bullet))$ and $\mathcal{H}^1(\fmd(\M^\bullet))$ along the components of their supports. However, in this process the sheaf $\mathcal{H}^0(\fmd\M^\bullet)$ is not changed, and it is always equal to the $\lambda$-torsion of $\mathcal{H}^1(\tfm(\M^\bullet))$. By keep doing this procedure for each irreducible component of the support of $\mathcal{H}^1(\tfm(\M^\bullet))$, eventually we get a complex of mixed Hodge modules $\M^\bullet$, with zero total Fourier transform (hence zero de Rham Fourier transform), but nonzero Dolbeault Fourier transform if $$\mathcal{H}^0(\fmd(\M^\bullet))\simeq \ker(\cdot\lambda:\mathcal{H}^1(\tfm(\M^\bullet))\rightarrow \mathcal{H}^1(\tfm(\M^\bullet)))$$ is nonempty. This yields a contradiction to the fact that objects in $\operatorname{D}^b_{\MHM}(A)$ have the same cohomology support loci on de Rham and Dolbeault moduli spaces under the identification by non-abelian Hodge correspondence.

We have proved that the cohomology sheaves of the total Fourier transform of a complex of mixed Hodge modules $\M^\bullet\in \operatorname{D}^b_{\MHM}(A)$ are flat over $\mathbb{A}^1_\lambda$.
    
\end{proof}

\begin{remark}
    Our main result \theoremref{thm: main theorem} is crucial for the above proof of flatness, since we need to ensure the supports of $\tfm$ are linear in $E(A)$, and then we can apply the functorial operations. However, in the last step, it's not necessary to invoke our main theorem, in fact $E_1$-degeneration for mixed Hodge modules is enough. Since the total transform is zero, we know the de Rham transform is also zero, hence there are no cohomology support loci on $A^\natural$. But since Dolbeault transform is nonzero, there has to be a support loci whose intersection with $\hat{A}\times \{0\}\subset \hat{A}\times V$ is nonempty by $\C^*$-invariance of $\fmd(\M)$. Then by $E_1$-degeneration, the de Rham cohomology support loci should be nonempty, which is a contradiction.
\end{remark}

\section{Other realizations for Tannaka groups of Hodge modules}

Recently there are various applications of generic vanishing to geometry of subvarieties of abelian varieties via representation theory \cite{TannakagroupgenericPPAV}, \cite{microlocalI}, \cite{microlocalII}. The connection to representation theory is because of the following. Given a Hodge module $\M$, there's a tensor category $\langle \M\rangle$ generated by $\M$, where the tensor product is given by the convolution product $\M_1 *\M_2 \defeq a_+(\M_1\boxtimes \M_2)$ ($a:A\times A\rightarrow A$ is the addition map, and one has to pass to a suitable quotient category to get a tensor category, see \cite[section 1.b]{microlocalI}). After appropriate reductions, one gets a neutral Tannakian category, with a fiber functor which is given by a general fiber of the Fourier transform of $\M$ on $A^\natural$. Associated to a Tannakian category with a fiber functor, there's an affine algebraic group whose category of representations is equivalent to the Tannakian category. In our setting, any fiber functor $\omega$ on $\langle \M\rangle$ induces an equivalence
$$\langle \M\rangle \xrightarrow{\sim} \operatorname{Rep}(G)$$
with the tensor category of finite dimensional algebraic representations of the affine algebraic group $G=\operatorname{Aut}^\otimes(\omega)$ of the tensor automorphisms of $\omega$.
By \cite{DeligneMilne}, on any Tannakian category over an algebraically closed field any two fibers functors are isomorphic, so the group $G=G_{\M}$ is uniquely determined up to isomorphism. Since $\M$ is semi-simple, $G$ is a reductive group \cite[Remark 5.4]{kramersemiabelian}. In \cite{microlocalI}, Krämer showed that the Tannaka group is also the minimal reductive structure group of the Fourier transform of $\M$ on the open set of $A^\natural$ where it's locally free.

For a principal $G$-bundle $\mathscr{G}$ on a smooth algebraic variety $X$, for any linear representation $G\rightarrow \operatorname{GL}(V)$, there's an associated vector bundle $\mathscr{G}\times^G V$ on $X$. This gives a $\C$-linear exact tensor functor 
$$\Phi_{\mathscr{G}}: \operatorname{Rep}(G)\rightarrow \operatorname{Coh}(X).$$
We say that an algebraic vector bundle over a smooth algebraic variety $X$ has a reduction to a principal $G$-bundle $\mathscr{G}$ if it is an associated vector bundle of $\mathscr{G}$ for some linear representation of $G$. Any vector bundle of rank $n$ has a reduction to a principal bundle with structure group $\mathrm{GL}_n(\mathbb{C})$.

\begin{definition}[Naturanl reduction for a vector bundle]
    A reduction of a given vector bundle to a principal bundle
$\mathscr{H}$ is called a \textit{natural reduction} if
\begin{itemize}
    \item[(1)]
    the structure group $H$ of the principal bundle $\mathscr{H}$
    is reductive, and

    \item[(2)]
    for any other reduction of the given vector bundle to a
    principal bundle $\mathscr{G}$ with a reductive structure
    group $G$ there exists an embedding
    $\iota : H \hookrightarrow G$ such that the following
    diagram commutes:
    $$
    \begin{tikzcd}[column sep=large, row sep=large]
    \operatorname{Rep}(G)
        \arrow[rr, "\Phi_{\mathscr{G}}"]
        \arrow[dr, dotted, "\iota^*"']
    && \operatorname{Coh}(\mathscr{O}_U) \\
    &
    \operatorname{Rep}(H)
        \arrow[ur, "\Phi_{\mathscr{H}}"']
    &
    \end{tikzcd} $$
\end{itemize}
\end{definition}

Bogomolov \cite{BogomolovStablebundle} showed that on a smooth variety $X$ with $H^0(X,\strucsheaf_X)=\C$, any algebraic vector bundle has a natural reduction. Krämer applied Bogomolov's result with the reconstruction property \cite[Corollary 22.3]{holonomicDmodules} of Fourier transform of holonomic $\D$-modules to show the following result.

\begin{theorem}[Theorem 3.2 of \cite{microlocalI}]
    For a semi-simple holonomic $\D$-module $\M$ on $A$, on an open subset $U$ of $A^\natural$, $\fm_A(\M)$ is induced by a principal $G$-bundle $\mathscr{G}$, where $G=G_\M$ is the Tannaka group of $\M$. Moreover $\mathscr{G}$ is a natural reduction of $\fm_A(\M)|_U$.
\end{theorem}

Applying our results, we get different fiber functors on the Tannakian category $\langle \M\rangle$. We will also show that on $\proj(E(A))$, the descended total Fourier transform have the Tannaka group $G_\M$ as their structure groups, and on $\proj(E(A))$ it's the minimal structure group.

\begin{proposition}
The functor $\fmd_A: \operatorname{D}^b_{\MHM}(A) \longrightarrow \operatorname{D}^b_{\mathrm{coh}}(\strucsheaf_{\hat{A}\times V})$ underlies a tensor functor with respect to the tensor structures given by the convolution product on the source and the usual derived tensor product on the target. The same thing is true for 
$\tfm_A: \operatorname{D}^b_{\MHM}(A) \longrightarrow \operatorname{D}^b_{\mathrm{coh}}(\strucsheaf_{E(A)})$.
\end{proposition}
\begin{proof}
    The proof works the same as \cite[Proposition 3.1]{microlocalI}, using the functoriality results \theoremref{thm:functoriality} and \theoremref{thm:extendedfunctoriality} for Fourier transform of mixed Hodge modules.
\end{proof}

Now we focus on Fourier transforms for pure Hodge modules. For the total Fourier transform, the total space of $E(A)$ is a vector bundle over $\hat{A}$, so it has a natural $\C^*$-action, and $\tfm_A(\M)$ is equivariant with respect to the $\C^*$-action, so it descends to a complex of coherent sheaves on the projectivized bundle $\proj(E(A))$. The descended total Fourier transform is still denoted by $\tfm_A(\M)$. By \corollaryref{cor: generic vanishing for Dolbeault cohomology}, there exists an open set $U\subset\proj(E(A))$, such that
\begin{equation} \label{eq: total transform locally free}
\mathcal{H}^{i}\bigl(\tfm_A(\M)\bigr)\big|_{U}
\quad \text{is} \quad
\begin{cases}
    \text{locally free} & \text{for } i=0,\\
    \text{zero}        & \text{for } i\neq 0.
\end{cases}
\end{equation}
Furthermore, after removing finitely many translates of linear subvarieties, we may assume all subquotients of $\M$ also satisfies property~\eqref{eq: total transform locally free} on $U$. 


We define $\mathrm{M}(A)$ be the
quotient of the category of Hodge modules by the Serre
subcategory of all Hodge modules with Euler characteristic zero (they are called negligible modules in \cite[section 3]{microlocalI}). For an open subset $U\subset \hat{A}\times V$ or $U\subset \proj(E(A))$, where $U=p^{-1}(U')$ for some $U'\subset \hat{A}\times V$, denote by
$\mathrm{M}(A,U)$ the quotient of the full abelian subcategory of Hodge modules with locally free Fourier transform on $U$ by the Serre subcategory of all negligible modules contained in it. As in \cite[section 3]{microlocalI}, the total Fourier transform induces tensor functors
$$\widetilde{\Phi}: \mathrm{M}(A,U)\rightarrow \operatorname{Coh}(\strucsheaf_U),\quad \M\rightarrow \mathcal{H}^0(\tfm_A(\M))|_U, \quad U\subset \proj(E(A)).$$
These tensor functors are exact and takes values in locally free sheaves on $U$ by \corollaryref{cor: generic vanishing for Dolbeault cohomology}.


\begin{theorem}
    For any $\M$, the vector bundle $\widetilde{\Phi}(\M) \in \operatorname{Coh}(\strucsheaf_U)$ on $U\subset \proj(E(A))$
is induced by a principal bundle $\mathscr{G} \longrightarrow U$ whose structure group is isomorphic to the Tannaka group $G = G_\M$. $\mathscr{G}$ is a natural reduction for this vector bundle.
\end{theorem}
\begin{proof}
    The same proof of \cite[Theorem 3.2]{microlocalI} works here in our setting. A splitting of $\widetilde{\phi}(\mathcal{N})$ for $\mathcal{N}\in \langle \M\rangle$ on $U\subset \proj(E(A))$ will induce a splitting on $U\cap A^\natural$. Because we can take $U$ to be complement of some translates of linear subvarieties, so codimension of complement of $U$ will be greater than or equal to two, hence So $H^0(U,\strucsheaf_U)=\C$. Bogomolov's theorem \cite[Theorem 2.1]{BogomolovStablebundle} applies, and we can proceed as in Krämer's proof.
\end{proof}




\section{A simple proof of the propagation property}
In this section, we give a simple proof of the following weaker version of the propagation property in \cite[Proposition 2.4]{perversesemiabelian}, but it suffices for our purpose. We also verify the propagation property for cohomology support loci of $\fmd_A(\M)$ for a mixed Hodge module $\M$.

Recall that the cohomology support loci $S^k(F^\bullet)$ of a bounded finitely generated complex of free $R$-modules to be the set of maximal ideals $m$ of $R$ such that $H^k(F^\bullet \otimes R/m)\neq 0$.

\begin{proposition}\label{prop:propagation}
    Let $F^\bullet$ be a bounded complex of finitely generated free $R$-modules, where $R$ is a regular noetherian ring. Suppose  $H^i(F^\bullet )$ and $H^i\left(\operatorname{Hom}_R\left(F^{\bullet}, R\right)\right)=0$, for any $i<0$.  Then the cohomology support loci of $F^\bullet$ satisfy the following propagation property:
    $$S^{-k}\left(F^{\bullet}\right) \subseteq \cdots \subseteq S^{-1}\left(F^{\bullet}\right) \subseteq S^0\left(F^{\bullet}\right) \supseteq S^1\left(F^{\bullet}\right) \supseteq \cdots \supseteq S^{\ell}\left(F^{\bullet}\right)$$
\end{proposition}
\begin{proof}
    We first localize at a closed point and assume $R$ is a local ring. Since any bounded complex of $R$-modules with finitely generated cohomology over a local ring is homotopy equivalent to a minimal complex, without loss of generality we assume $F^\bullet$ is also a minimal complex. Here a minimal complex means a bounded complex of free modules with differentials having coefficients in the maximal ideal $m$.

    Suppose $i<0$. We want to show if $H^{i}(F^\bullet \otimes R/m)=0$, then $H^{i-1}(F^\bullet\otimes R/m)=0$. Since the coefficients of $F^\bullet$ are in the maximal ideal $m$, then $\dim_{R/m} H^{k}(F^\bullet \otimes R/m) = \operatorname{rank} F^k$ for any $k$. Hence $F^i=0$ by Nakayama's lemma. Since $H^i(F^\bullet)=0$ by assumption, we can further replace $F^{k}$ where $k<i$ by zero $R$-modules and still get a quasi-isomorphic complex. It then trivially follows that $H^{i-1}(F^\bullet\otimes R/m)=0$.

    Now let $i>0$, we want to show if $H^{i}(F^\bullet \otimes R/m)=0$, then $H^{i+1}(F^\bullet\otimes R/m)=0$. Because $H^i(F^\bullet \otimes R/m)=0$, we have $F^i=0$ by Nakayama's lemma. Same way as in the last paragraph, we can replace $\operatorname{Hom}_R(F^\bullet,R)$ by zero modules in degree less than $-i$, and still get a quasi-isomorphic complex. Apply $\operatorname{Hom}_R$ functor again, it trivially follows that $H^{i+1}(R^\bullet \otimes R/m)$ is zero.
    
\end{proof}

We also verify that our Fourier-Mukai transform $\fmd_A(\M)$ for a mixed Hodge module $\M$ satisfies the assumptions in this proposition. In the de Rham version of the Fourier transform, this is proved in \cite{mellinpropagation}. The first assumption on $\fmd_A(\M)$ is automatic since it's the pushforward of a single sheaf, so there's cohomologies in negative degrees. The second assumption is true by part $(b)$ of \theoremref{thm:functoriality}. 
\begin{corollary}\label{cor:propagationDolbeault}
    On an abelian variety $A$ of dimension $g$, the cohomology support loci of $\fmd_A(\M)$ for a mixed Hodge module $\M$ satisfy the propagation property,
    $$\emptyset=S^{g+1}\subseteq S^g\subseteq \cdots\subseteq S^1\subseteq S^0\supseteq S^{-1}\supseteq \cdots \supseteq S^{-g}\supseteq S^{-g-1}=\emptyset.$$
\end{corollary}

\newpage
\bibliographystyle{alpha}
\bibliography{references.bib}

\end{document}